\documentclass[11pt,reqno]{amsart}
\usepackage{amsmath, amssymb, latexsym, enumerate,  graphicx}
\usepackage{amscd,mathrsfs,epic,empheq,float}
\usepackage{etex}
\usepackage{bbm}
\usepackage[all]{xy}
\usepackage{pdfsync}
\usepackage{tikz}
\usetikzlibrary{arrows}
\usepackage{graphicx}
\usepackage{wasysym}
 \newlength{\baseunit}               
 \newcount{\numlines}                
\newtheorem{theorem}{Theorem}
\newtheorem{lemma}[theorem]{Lemma}
\newtheorem{remark}[theorem]{Remark}
\newtheorem{proposition}[theorem]{Proposition}
\newtheorem{corollary}[theorem]{Corollary}
\newtheorem{example}[theorem]{Example}

\newtheorem{definition}[theorem]{Definition}

\newtheorem{claim}{Claim}

\newcommand{\cO}{{\mathcal O}}

\newcommand{\Br}{\operatorname{Br}}

\def\down{\vee}
\def\up{\wedge}

\newcommand{\mC}{\mathbb{C}}

\newcommand{\mZ}{\mathbb{Z}}

\newcommand{\op}{\operatorname}

\DeclareMathOperator{\Hom}{Hom}

\makeatletter
\def\Set#1{\Set@h#1@}
\def\Set@h#1|#2@{\left\{\left.#1\vphantom{#2}\hskip.1em\,\right|\,\relax #2\right\}}
\makeatother

\begin{document}

\title[Schur-Weyl duality for the Brauer algebra]{Schur-Weyl duality for the Brauer algebra and the ortho-symplectic Lie superalgebra}

\author{Michael Ehrig \and Catharina Stroppel}

\address{Michael Ehrig\\ Endenicher Allee 60, 53115 Bonn, Germany}
\address{Catharina Stroppel\\ Endenicher Allee 60, 53115 Bonn, Germany}

\begin{abstract}
We give a proof of a Schur-Weyl duality statement between the Brauer algebra and the ortho-symplectic Lie superalgebra $\mathfrak{osp}(V)$.
\end{abstract}

\thanks{M.E. was financed by the Deutsche Forschungsgemeinschaft Priority program 1388, C.S. thanks the MPI in Bonn for excellent working conditionss and financial support.
}

\maketitle

\section{Introduction}

\renewcommand{\thetheorem}{\Alph{theorem}}
In this paper we describe the centralizer of the action of the ortho-symplectic Lie superalgebra $\mathfrak{osp}(V)$ on the tensor powers of its defining representation $V$.  Here $V$ is a vector superspace of superdimension ${\rm sdim}V$ equal to $2m|2n$ or $2m+1|2n$ equipped with a supersymmetric bilinear form, and $\mathfrak{osp}(V)$ denotes the Lie superalgebra of all endomorphisms preserving this form. In particular, the extremal cases $n=0$ respectively $m=0$ give the classical orthogonal respectively symplectic simple Lie algebras. Our main result is the following generalization of Brauer's centralizer theorem in the classical case, see e.g. \cite{GW}, to a Lie superalgebra version. 
\begin{theorem} \label{thm:iso}
Let $m$ and $n$ be nonnegative integers and let $V$ be as above. Let $\delta=2m-2n$ respectively $\delta=2m+1-2n$ be the supertrace of $V$. Assume that one of the following assumptions holds
\begin{itemize}
\item ${\rm sdim} V \neq 2m|0$ and $d \leq m+n$ or
\item ${\rm sdim} V = 2m|0$ with $m > 0$  and $d < m$.
\end{itemize}
Then there is a canonical isomorphism of algebras
\begin{eqnarray}
\label{theoAiso}
{\rm End}_{\mathfrak{osp}(V)}(V^{\otimes d}) &\cong& {\rm Br}_d(\delta).
\end{eqnarray}
\end{theorem}
Here ${\rm Br}_d(\delta)$ denotes the Brauer algebra on $d$ strands with parameter $\delta$ which was originally introduced by Brauer in \cite{Brauer}, see Definition \ref{defBrauer}. Note that in case $m=0$ or $n=0$ we obtain the Lie algebra version of Brauer's classical centralizer theorem explained in modern language for instance in \cite{GW}. In contrast to these classical cases, the endomorphism algebras of finite dimensional representations of a Lie superalgebra are in general not semsimple. Hence, our theorem covers also the (most interesting) cases in which ${\rm Br}_d(\delta)$ is not semisimple.  These non-semisimple Brauer algebras also appear as idempotent truncations of endomorphism rings of modules in category $\cO$ for the classical Lie algebras, see \cite{ES}, \cite{ES2}. In \cite{ES} it was shown that the Brauer algebra can be equipped with a nonnegative $\mZ$-grading (which is even Koszul in case $\delta\not=0)$, see \cite{ES}, \cite{ES2}, \cite{GeLi}. As a consequence of our theorem, the endomorphism rings \eqref{theoAiso} inherit a grading. The existence of such a grading on Brauer algebras is rather unexpected and up to now this grading cannot be described intrinsically in terms of the representation theory of the Lie superalgebra.

The existence of a ring homomorphism \eqref{theoAiso} from the Brauer algebra to the endomorphism ring of $V^{\otimes d}$ is not new.  It is for instance a special case of the Brauer algebra actions defined in \cite{Benkartetal} and also follows from the more abstract theory of Deligne categories, \cite{Deligne}. However, the isomorphism theorem is, as far as we know, new. The crucial point here is that in general (and in particular in contrast to the classical case),  the tensor space is not completely reducible and then unfortunately the methods from  \cite{GW} and \cite{Benkartetal} do not apply. As far as we know there is not much known about the indecomposable summands appearing in this tensor space and even less is known about the general structure of the category of finite dimensional representations of  $\mathfrak{osp}(V)$, see \cite{GrusonSerganova}, \cite{Serganova}, and also \cite{Ke} for some partial results. This is very much in contrast to the case of the general Lie superalgebra $\mathfrak{gl}(m|n)$, where the corresponding (mixed) tensor spaces were recently studied and described in detail, \cite{CW}, \cite{BS5}. There, the role of the Brauer algebra is played by the {\it walled} Brauer algebra respectively {\it walled} Brauer category, sometimes also called oriented Brauer category, see \cite{BS5},  \cite{Betal}, \cite{SS}. As our main tool we construct an embedding of the Brauer algebra into an additive closure of the walled Brauer category and then use the above mentioned results for the general linear case to deduce Theorem~\ref{thm:iso} in the ortho-symplectic case.

More precisely we restrict the action of $\mathfrak{osp}(V)$ to a suitably embedded $\mathfrak{gl}(m|n)$ and decompose the occurring representation $V^{\otimes d}$ as a  direct sum of mixed tensor products on which the walled Brauer category acts naturally. One of the crucial steps is the following generalization from \cite{BS5}:

\begin{theorem} \label{thm:obrauer_action}
Let $m,n\in\mZ_{\geq 0}$, $d>0$. The corresponding oriented Brauer category $\mathcal{OB}_d(m-n)$ acts on $V^{\otimes d}$ and its action commutes with the action of $\mathfrak{gl}(m|n)$.
\end{theorem}

We then embed the Brauer algebra into the endomorphism ring of $V^{\otimes d}$ viewed as an object inside the additive closure of the oriented Brauer category. The injectivity of Theorem \ref{thm:iso} is then deduced from  the corresponding injectivity result  \cite[Theorem 7.1]{BS5} for walled Brauer algebras. Finally we show (by very elementary arguments) that any  $\mathfrak{gl}(m|n)$ endomorphism that commutes also with the action of $\mathfrak{osp}(V)$ is already contained in the image of the embedding and gives the desired isomorphism. (This method of proof does not give optimal bounds. Optimal bounds could be established recently in \cite{LZ3}, however on the cost of using less elementary arguments.)

We also like to stress that our proof does not rely on a Schur-Weyl duality for the ortho-symplectic super{\it group} as considered independently by Lehrer and Zhang in \cite{LZ1} using methods from graded-commutative algebraic geometry.  Our proof of the isomorphism theorem is considerably more elementary and only relies on a rather simple reduction to the general linear Lie superalgebra case, see \eqref{THEdiagram} and the mixed Schur-Weyl duality from \cite{BS5} which in turn is deduced from the Schur-Weyl duality of Sergeev \cite{Sergeev} and Berele-Regev \cite{BR} for the general linear Lie superalgebras. A different reduction argument (to the general linear group) was recently used in \cite{DLZ}.  Using however the non-elementary grading results of \cite{BS5}, it is possible to describe also the kernel of our action in general, but it requires to work with a graded version of the Brauer algebra as defined in \cite{ES} and will appear in a separate article where this graded Brauer algebra is studied. The extra grading refines the results from \cite{LZ1}, \cite{LZ2}.  
We expect that our approach generalizes easily to the quantized (super) case using the quantised walled Brauer algebras \hfill\\

\textbf{Acknowledgement} We thank Gus Lehrer for pointing out an inaccuracy in an earlier version of the paper as well as Kevin Coulembier and Daniel Tubbenhauer for various comments on a preliminary version.

\section{The Lie superalgebra}
\renewcommand{\thetheorem}{\arabic{section}.\arabic{theorem}}

We fix as ground field the complex numbers $\mathbb{C}$. By a superspace we mean a $\mZ/2\mZ$ -graded vector space $V = V_0 \oplus V_1$; for  any homogeneous element $v\in V$ we denote by $| v | \in \{0,1 \}$ its parity.  The integer $\op{dim}V_0-\op{dim}V_1$ is called the {\it supertrace} of $V$ and we denote by ${\rm sdim}V = \op{dim}V_0|\op{dim}V_1$ its {\it superdimension}. Given a superspace $V$ let $\mathfrak{gl}(V)$ be the corresponding general Lie superalgebra, i.e. the superspace ${\rm End}_\mathbb{C}(V)$ of all endomorphism  with the superbracket
$$[X,Y] = X \circ Y - (-1)^{|X| \cdot |Y|}Y \circ X .$$

For the whole paper we fix now  $n\in\mathbb{Z}_{\geq0}$ and a finite dimensional superspace $V = V_0 \oplus V_1$ with $\op{dim}V_1=2n$, and  equipped with a non-degenerate supersymmetric bilinear form $\left\langle - , - \right\rangle$, i.e. a bilinear form $V\times V\rightarrow \mC$ which is symmetric when restricted to $V_0 \times V_0$, skew-symmetric on $V_1 \times V_1$ and zero on mixed products. It will sometimes be convenient to work with a fixed homogeneous basis $v_i, i \in I$ of $V$, for a suitable indexing set $I$, and right dual basis $v_i^*$, i.e.  $\left\langle v_i,v_j^* \right\rangle = \delta_{i,j}$. Attached to this data we have the following:

\begin{definition}
The \emph{ortho-symplectic Lie superalgebra} $\mathfrak{osp}(V)$ is the Lie supersubalgebra of $\mathfrak{gl}(V)$ consisting of all endomorphisms which respect the supersymmetric bilinear form. Explicitly, a homogeneous element $X \in \mathfrak{osp}(V)$ has to satisfy 
\begin{equation}
\label{oh}
\left\langle Xv,w \right\rangle + (-1)^{|X| \cdot |v|} \left\langle v,Xw\right\rangle = 0,
\end{equation}
for homogeneous $v\in V$.
\end{definition}
From now on we make the convention that, whenever the parity of an element appears in a formula, the element is assumed to be homogeneous.

\begin{remark}{\rm
If $X \in \mathfrak{osp}(V)$, then $\left\langle Xv,w \right\rangle \neq 0$ implies $|X| = |v| + |w|$.
}
\end{remark}

\section{The Brauer algebra}

The following algebra was originally defined by Brauer \cite{Brauer} in his study of the orthogonal group.

\begin{definition}
\label{defBrauer}
Let $d\in\mZ_{\geq 0}$ and $\delta\in\mC$. The \emph{Brauer algebra} $\Br_d(\delta)$ is the associative unital $\mC$-algebra generated by the elements 
$$s_i,\, e_i \text{ for } 1\leq i\leq d-1,$$
subject to the relations
\begin{eqnarray}
\label{relBrauer}
&s_i^2=1, \quad s_i s_j=s_j s_i,\quad s_k s_{k+1} s_k=s_{k+1} s_k s_{k+1},&\nonumber\\
&e_i^2=\delta e_i, \quad e_i e_j=e_j e_i,\quad e_k e_{k+1} e_k=e_{k+1}, \quad e_{k+1} e_{k} e_{k+1}=e_{k},&\\
&s_i e_i=e_i= e_i s_i,\quad s_k e_{k+1} e_k=s_{k+1} e_k, \quad s_{k+1} e_{k} e_{k+1}=s_{k}e_{k+1},&\nonumber
\end{eqnarray}
for $1 \leq i,j \leq d-1$, $\mid i-j \mid > 1$, and $1 \leq k \leq d-2$.
\end{definition}

Following Brauer we realize $\Br_d(\delta)$ as a diagram algebra: A \emph{Brauer diagram} on $2d$ vertices is a partitioning $b$ of the set $\{1,2,\ldots, d, 1^*,2^*,\ldots d^*\}$ into $d$ subsets of cardinality $2$. Let $\mathcal{B}[d]$ be the set of such Brauer diagrams. An element can be displayed graphically by arranging $2d$ vertices in two rows $1,2,\ldots, d$ and ${1^*,2^*,\dots, d^*}$, with each vertex linked to precisely one other vertex. Two such diagrams are considered to be the same or equivalent if they link the same $d$ pairs of points. 

Special Brauer diagrams are the "unit" $1=\{\{1,1^*\},\{2,2^*\},\cdots, \{d,d^*\}\}$ connecting always $j$ with $j^*$ for all $1\leq j\leq r$, and for $1\leq i\leq r-1$ the $s_i$ (respectively $e_i$) which connects $j$ with $j^*$ except of the new pairs $\{i,(i+1)^*\}, \{i+1,i^*\}$ (respectively $\{i,i+1\},\{i^*,(i+1)^*\}$) involving the numbers $i$ and $i+1$;
\begin{equation}
\label{diag}
\begin{tikzpicture}[scale=0.7,thick,>=angle 90]
\node at (0,.5) {$s_i$};
\draw (.6,0) -- +(0,1);
\draw [dotted] (1,.5) -- +(1,0);
\draw (2.4,0) -- +(0,1);
\draw (3,0) -- +(.6,1) node[above] {\tiny i+1};
\draw (3.6,0) -- +(-.6,1) node[above] {\tiny i};
\draw (4.2,0) -- +(0,1);
\draw [dotted] (4.6,.5) -- +(1,0);
\draw (6.2,0) -- +(0,1);

\begin{scope}[xshift=8cm]
\node at (0,.5) {$e_i$};
\draw (.6,0) -- +(0,1);
\draw [dotted] (1,.5) -- +(1,0);
\draw (2.4,0) -- +(0,1);
\draw (3,0) to [out=90,in=-180] +(.3,.3) to [out=0,in=90] +(.3,-.3);
\draw (3,1) node[above] {\tiny i} to [out=-90,in=-180] +(.3,-.3) to [out=0,in=-90] +(.3,.3) node[above] {\tiny i+1};
\draw (4.2,0) -- +(0,1);
\draw [dotted] (4.6,.5) -- +(1,0);
\draw (6.2,0) -- +(0,1);
\end{scope}
\end{tikzpicture}
\end{equation}

Given two Brauer diagrams $b$ and $b'$, their concatenation $b'\circ b$ is obtained by putting $b$ on top of $b'$ identifying vertex $i^*$ in $b$ with vertex $i$ in $b'$ and removing all the internal loops. Let $c(b,b')$ be the number of loops removed.  Then we have the following fact, see e.g. \cite[Section 9 and 10]{GW} for details:

\begin{lemma}
The \emph{Brauer algebra} $\Br_d(\delta)$ is, via the assignment \eqref{diag} on generators, canonically isomorphic to the $\mC$-algebra with basis $\mathcal{B}[d]$ and multiplication $bb'=\delta^{c(b,b')} b\circ b'$ for $b,b'\in\mathcal{B}[d]$.
\end{lemma}

\begin{remark}{\rm
Generically, the algebra $\Br_d(\delta)$ is semi-simple, but not semi-simple for specific integral values for $\delta$ (dependent on $d$), see \cite{Wenzl}, \cite{Rui} and also \cite{AST}. For a detailed study of the non semi-simple algebras over the complex numbers we refer to \cite[2.2]{ES}.}
\end{remark}

We now define the two actions on the tensor space $V^{\otimes d}$ from Theorem~\ref{thm:iso}. The action of $\mathfrak{osp}(V)$ on $V^{\otimes d}$ is given by the comultiplication 
$$\Delta (X) = X \otimes 1 + 1 \otimes X$$ for $X \in \mathfrak{osp}(V)$, keeping in mind the tensor product rule $$(X \otimes Y)(v \otimes w) = (-1)^{|Y| \cdot |v|}(Xv \otimes Yw).$$
Explicitly we have for homogeneous elements $w_i\in V$ 
\[ 
X.(w_1 {\scriptstyle \otimes} \ldots {\scriptstyle \otimes} w_r) = \sum_{i=1}^r (-1)^{\left( \sum_{j=1}^{i-1}|w_j|\right)|X|} w_1 {\scriptstyle \otimes} \ldots {\scriptstyle \otimes} w_{i-1} {\scriptstyle \otimes} Xw_i {\scriptstyle \otimes} w_{i+1} {\scriptstyle \otimes} \ldots {\scriptstyle \otimes} w_r.
\]

\begin{definition}
\label{eisi}
Define the following linear endomorphisms 
\begin{eqnarray*}
\sigma :\, V \otimes V &\longrightarrow& V \otimes V \text{, via } \sigma (v \otimes w) := (-1)^{|v| \cdot |w|} w \otimes v \text{ and }\\
\tau : \, V \otimes V &\longrightarrow& V \otimes V \text{, via } \tau (v \otimes w) := \left\langle v,w\right\rangle \sum_{i \in I} (-1)^{|v_i|} v_i \otimes v_i^*.
\end{eqnarray*}
of $V \otimes V$ and for each fixed natural number $d\geq 2$ the endomorphisms 
\begin{eqnarray*}
s_i = {\rm id}^{\otimes (i-1)} \otimes \sigma \otimes {\rm id}^{\otimes (d-i-1)}
&\text{and}&e_i = {\rm id}^{\otimes (i-1)} \otimes \tau \otimes {\rm id}^{\otimes (d-i-1)}
\end{eqnarray*}
of $V^{\otimes d}$ for $1 \leq i < d$.
\end{definition}

The following proposition is straight-forward to check.

\begin{proposition}
\label{equi}
The maps $\sigma, \tau$ and $s_i, e_i$ are $\mathfrak{osp}(V)$-equivariant.
\end{proposition}

\begin{theorem} \label{thm:Brauer_action}
Let $\delta = {\rm dim} V_0 - {\rm dim} V_1 $ be the supertrace of $V$. Then there is a right action of ${\rm Br}_d(\delta)$ on $V^{\otimes d}$ given by
\begin{eqnarray*}
v.s_i := s_i(v) &\text{and}& v.e_i = e_i(v)
\end{eqnarray*}
for $v\in V^{\otimes d}$ and $1\leq i\leq d-1$. It commutes with the action of $\mathfrak{osp}(V)$.
\end{theorem}

\begin{remark} \label{rem:Sd_action}{\rm
Restricting only to the $s_i$ gives us an action of the symmetric group $S_d$ on $V^{\otimes d}$, which clearly commutes with the action of $\mathfrak{gl}(V)$.  This is the super Schur-Weyl duality,  \cite[Theorem 7.5]{BS5},   originally proved by Sergeev \cite{Sergeev} and Berele-Regev, \cite{BR}.}
\end{remark}

\begin{proof}[Proof of Theorem \ref{thm:Brauer_action}]
By Proposition \ref{equi} all involved morphisms commute with the action of $\mathfrak{osp}(V)$. Hence it remains to show that it defines an action of the Brauer algebra. It is obvious that the $s_i$ define an action of the symmetric group and that the relations $e_ie_j=e_je_i$ are satisfied for $|i-j|>1$. Moreover, $e_i^2 = \delta e_i$ by definition. To  verify the relation $s_i e_i = e_i$ it suffices to assume $d=2$. Let $v,w \in V$ be homogeneous then
\begin{eqnarray*}
(v \otimes w).s_ie_i &=& (-1)^{|v| \cdot |w|} (w \otimes v).e_i =
(-1)^{|v| \cdot |w|} \left\langle w,v \right\rangle \sum_{l \in I}
(-1)^{|v_l|} v_l \otimes v_l^* \\
&=& \left\langle v,w \right\rangle \sum_{l \in I} (-1)^{|v_l|} v_l \otimes
v_l^* = (v \otimes w).e_i.
\end{eqnarray*}
For the equality $e_i = e_i s_i$ we note that $(v_j \otimes v_k^*).e_i =
0$ unless $j=k$, so we assume equality and obtain
\begin{eqnarray*}
(v_j \otimes v_j^*).e_is_i &=& \sum_{l \in I} (-1)^{|v_l|} (v_l \otimes
v_l^*).s_i 
= \sum_{l \in I} (-1)^{|v_l|}(-1)^{|v_l|} v_l^* \otimes v_l\\
&=& \sum_{l \in I} (-1)^{|v_l|} ((-1)^{|v_l|} v_l^*) \otimes v_l
= \sum_{l \in I} (-1)^{|v_l|} v_l \otimes v_l^*,
\end{eqnarray*}
where the final equality holds because the $\left\lbrace (-1)^{|v_l|}
v_l^* \right\rbrace$ also form a basis of $V$ with right dual basis
$\{v_l\}$.

For the final relations
\begin{eqnarray}
e_i e_{i+1} e_i (x)= e_{i}(x) &\text{ and }& e_{i+1} e_{i} e_{i+1}(x) = e_{i+1}(x),\label{first}\\
e_i e_{i+1} s_i (x) = e_{i}s_{i+1}(x) &\text{ and }& e_{i+1} e_{i} s_{i+1} (x) = e_{i+1}s_i(x)\label{second}
\end{eqnarray}
for $x \in V^{\otimes d}$, it suffices to consider the case $d=3$ and $x=u \otimes v \otimes w$ for homogeneous $u,v,w \in V$. The exact calculation is straight-forward and left to the reader.
\end{proof}

\section{The oriented Brauer Category}

We now recall the oriented Brauer category (or sometimes called walled Brauer category, see e.g. \cite{SS}), which has as objects certain orientation sequences and morphism spaces given by oriented generalized Brauer diagrams. We then relate it to the space $V^{\otimes d}$.

We set ${\rm OSeq}[d] = \{\up,\down\}^{\times d}$ and $\widehat{\rm OSeq}[d] = \{\up,\down,\circ\}^{\times d}$ and call its elements \emph{orientations} of length $d$ respectively \emph{generalized orientations} of length $d$. Let $\widehat{\mathcal{B}}[d]$ be the set of \emph{generalized Brauer diagrams} on $2d$ vertices,  that means diagrams which are Brauer diagrams, except that the partitioning is into subsets of cardinality 1 or 2. In other words, we allow vertices that are not connected to any other vertex.

\begin{definition}
\label{defoje}
An \emph{oriented generalized Brauer diagram} is a triple $(\mathbf{t}, b, \mathbf{s})$ where $\mathbf{s},\mathbf{t} \in \widehat{\rm OSeq}[d]$ and $b \in \widehat{\mathcal{B}}[d]$ such that the following conditions hold:
\begin{enumerate}
\item if $\{i,j\} \in b$ with $i \neq j$ then $\{\mathbf{t}_i,\mathbf{t}_j\} = \{\up,\down\}$,
\item if $\{i^*,j^*\} \in b$ with $i \neq j$ then $\{\mathbf{s}_i,\mathbf{s}_j\} = \{\up,\down\}$,
\item if $\{i,j^*\} \in b$ then $\{\mathbf{t}_i,\mathbf{s}_j\} \in \{ \{\up\}, \{\down\}\}$, and
\item if $\{i\} \in b$ or  $\{i^*\} \in b$, then $\mathbf{t}_i = \circ$, respectively $\mathbf{s}_i=\circ$.
\end{enumerate}
The pair $(\mathbf{s},\mathbf{t})$ is then called the \emph{orientation}.
\end{definition}
We denote the set of diagrams oriented with orientation $(\mathbf{s}, \mathbf{t})$ by $\widehat{\mathcal{B}}_{\mathbf{s}}^{\mathbf{t}}$ or by $\widehat{\mathcal{B}}[d]_{\mathbf{s}}^{\mathbf{t}}$ if we want to stress that the length is $d$, and draw them as oriented diagrams (in the obvious sense) with the orientation sequence $\mathbf{s}$ at the bottom and the orientation sequence $\mathbf{t}$ on the top. 

It will be sometimes helpful to vary $d$ and so we set $\widehat{\rm OSeq}=\bigcup_{d'\geq 0}\widehat{\rm OSeq}[d']$ and ${\rm OSeq}=\bigcup_{d'\geq 0}\widehat{\rm OSeq}[d'] $. For $\mathbf{s}\in \widehat{\rm OSeq}$ we  denote by $\overline{\mathbf{s}}\in {\rm OSeq}\subset \widehat{\rm OSeq}$ the associated \emph{reduced sequence} obtained by deleting all $\circ$'s. Similarly define $\overline{b}$ for  $b\in\widehat{\mathcal{B}}[d]$. Note that if  $b\in\widehat{\mathcal{B}}_{\mathbf{s}}^{\mathbf{t}}$ then $\overline{b}\in\widehat{\mathcal{B}}_{\overline{\mathbf{s}}}^{\overline{\mathbf{t}}}$. Moreover, we can allow pairs of sequences $(\mathbf{s},\mathbf{t})$  not of the same length and define the sets $\widehat{\mathcal{B}}_{\mathbf{s}}^{\mathbf{t}}$ and $\widehat{\mathcal{B}}_{\overline{\mathbf{s}}}^{\overline{\mathbf{t}}}$ of (reduced) oriented generalized Brauer diagrams with orientation $(\mathbf{s},\mathbf{t})$ exactly as in Definition \ref{defoje}.  

\begin{example} The first diagram below is an element in $\widehat{\mathcal{B}}[4]_{(\down,\circ,\up,\down)}^{(\down,\up,\down,\circ)}$. The other two diagrams are not oriented. In both cases the pieces highlighted by a dashed line violate the orientation conditions.
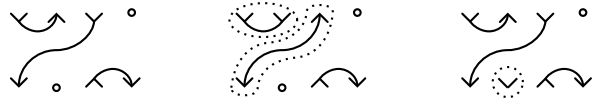
\begin{eqnarray*}
\begin{tikzpicture}[scale=0.5,thick,>=angle 90]
\draw [>->] (0,2) to [out=-90,in=-180] +(.5,-.5) to [out=0,in=-90] +(.5,.5);
\draw [<-<] (0,0) to [out=90,in=-180] +(1,1) to [out=0,in=-90] +(1,1);
\draw [>->] (2,0) to [out=90,in=-180] +(.5,.5) to [out=0,in=90] +(.5,-.5);
\draw (1,0) circle(2.5pt);
\draw (3,2) circle(2.5pt);

\begin{scope}[xshift=6cm]
\draw [>-<] (0,2) to [out=-90,in=-180] +(.5,-.5) to [out=0,in=-90]
+(.5,.5);
\draw [dotted] (.5,1.8) ellipse (.9 and .45);
\draw [dotted] (-.35,0) to [out=85,in=-175] +(1,1.2) to [out=5,in=-95] +(1,.8) to [out=90,in=-180] +(.35,.2) to [out=0,in=90] +(.35,-.2) to [out=-95,in=5] +(-1,-1.2) to [out=-175,in=85] +(-1,-0.8) to [out=-90,in=0] +(-.35,-.2) to [out=-180,in=-90] +(-.35,.2);
\draw [<->] (0,0) to [out=90,in=-180] +(1,1) to [out=0,in=-90] +(1,1);
\draw [>->] (2,0) to [out=90,in=-180] +(.5,.5) to [out=0,in=90] +(.5,-.5);
\draw (1,0) circle(2.5pt);
\draw (3,2) circle(2.5pt);
\end{scope}

\begin{scope}[xshift=12cm]
\draw [>->] (0,2) to [out=-90,in=-180] +(.5,-.5) to [out=0,in=-90] +(.5,.5);
\draw [<-<] (0,0) to [out=90,in=-180] +(1,1) to [out=0,in=-90] +(1,1);
\draw [>->] (2,0) to [out=90,in=-180] +(.5,.5) to [out=0,in=90] +(.5,-.5);
\draw [-] (1,0) -- +(0.25,0.25);
\draw [-] (1,0) -- +(-0.25,0.25);
\draw [dotted] (1,.15) circle(0.4);
\draw (3,2) circle(2.5pt);
\end{scope}
\end{tikzpicture}
\end{eqnarray*}
If one removes the highlighted parts of the diagrams above one obtains elements in $\widehat{\mathcal{B}}_{(\circ,\up,\down)}^{(\circ)}$ (for the second diagram) and  $\widehat{\mathcal{B}}_{(\down,\up,\down)}^{(\down,\up,\down,\circ)}$ (for the third diagram). 
\end{example}

\begin{definition} \label{def:walled_Brauer}
Let $d \in \mathbb{Z}_{\geq 0}$ and $\delta \in \mathbb{C}$. We define the \emph{oriented Brauer category} $\mathcal{OB}_d(m-n)$ as the following $\mathbb{C}$-linear category: The set of objects is $\widehat{\rm OSeq}[d]$, the morphism space $\Hom (\mathbf{s},\mathbf{t})$ for $\mathbf{s},\mathbf{t} \in \widehat{\rm OSeq}[d]$ is the vector space with basis $\widehat{\mathcal{B}}[d]_{\mathbf{s}}^{\mathbf{t}}$, and the composition of morphisms 
$$\Hom (\mathbf{s},\mathbf{t}) \times \Hom (\mathbf{r},\mathbf{s})\longrightarrow \Hom (\mathbf{r},\mathbf{t}),$$
is done on basis vectors by putting the two diagrams on top of each other, glueing along the entries of $\mathbf{s}$ and eliminating all internal $\circ$'s and also inner circles, each elimination of an internal circle resulting in multiplying the diagram with a factor of $m-n$. Similarly, define the category $\mathcal{OB}(m-n)$ by allowing as objects sequences of arbitrary finite length.
\end{definition}

\begin{lemma} \label{lem:hom_trivial}
Let $\mathbf{s},\mathbf{t} \in \widehat{\rm OSeq}[d]$ and $\ell = | \{i \mid \mathbf{s}_i = \circ\}| - |\{i \mid \mathbf{t}_i = \circ\}|$. Then $\Hom (\mathbf{s},\mathbf{t})=\{0\}$ if $\ell$ is odd or $\ell$ is even and additionally it holds
$$ \frac{\ell}{2}  \neq |\{i \mid \mathbf{s}_i = \up\}| - |\{i \mid \mathbf{t}_i = \up\}| .$$
\end{lemma}
\begin{proof}
This follows directly from the definitions.
\end{proof}

\begin{remark}
\label{Jon}
{\rm We like to stress that our category $\mathcal{OB}(m-n)$ is slightly different from the category called oriented walled Brauer category in \cite{Betal} in the sense that we have more objects, since we also fix the places of the trivial tensor factors. Hence, the oriented Brauer category from \cite{Betal} is the full subcategory of ours, where we only allow objects without $\circ$'s.}
\end{remark}

To relate this category to the endomorphisms of the superspace $V^{\otimes d}$ let $m$ be such that ${\rm dim}(V_0)=2m$ or ${\rm dim}(V_0)=2m+1$.  Unfortunately we have to distinguish between these two cases. Therefore, we will refer to them as the \emph{even case} respectively \emph{odd case}. Fix the sets
\begin{equation*}
I^\up = \{i \mid 1 \leq i \leq m+n\} \text{, } I^\down = \{\overline{i} \mid 1 \leq i \leq m+n\} \text{, } I^\circ = \{0\}, 
 \end{equation*}
and also  $I = I ^\up \cup I^\down$ in the even case and $I = I ^\up \cup I^\down \cup I^\circ$ in the odd case. Denote by $\mid \mid i \mid \mid$ the absolute value for an element in $i\in I$, in particular $\mid \mid \overline{i} \mid \mid = i=\mid \mid {i} \mid \mid $ for $\overline{i} \in I^\down$. To simplify calculations later on we use an explicit realization  of $\mathfrak{osp}(V)$ in terms of matrices, see e.g. \cite{Musson}.

\begin{definition}
\label{osp}
Let $V=\mC^{2m+1|2n}$ be the superspace with basis $v_i$ for $i \in I$ and $| v_i | = 0$ if $\mid \mid i \mid \mid \leq m$ and $| v_i | = 1$ otherwise. To write down matrices we order the basis elements as follows
$$ 0 < \overline{1}< {\scriptstyle \ldots} < \overline{m}<1 < {\scriptstyle \ldots} < m < \overline{m+1} < {\scriptstyle \ldots} < \overline{m+n} < m+1 < {\scriptstyle \ldots} < m+n.$$
We denote by  $\mathfrak{osp}(2m+1|2n)\subset \mathfrak{gl}(V)$ the Lie supersubalgebra given by 
\begin{equation*}
\renewcommand*\arraystretch{1.3} 
\Set{\left(\begin{array}{ccc|cc}
0 & -u_1^T & -u_2^T & x_1 & x_2 \\
u_2 & -A^T & a_1 & C^T & z_1 \\
u_1 & a_2 & A & z_2 & B \\
\hline
-x_2^T & -B^T & -z_1^T & -D^T & d_1\\ [1ex]
x_1^T & z_2^T & C & d_2 & D
\end{array}\right) | \begin{array}{l}
A,a_1,a_2 \text{ are } m \times m \text{ matrices }\\
D,d_1,d_2 \text{ are } n \times n \text{ matrices }\\
u_1,u_2 \text{ are } m \times 1 \text{ matrices }\\
x_1,x_2 \text{ are } 1 \times n \text{ matrices }\\
B,z_1,z_2 \text{ are } m \times n \text{ matrices }\\
C \text{ is a } n \times m \text{ matrix}\\
a_1,a_2 \text{ are skew-symmetric}\\
d_1,d_2 \text{ are symmetric}
\end{array}}.
\end{equation*}
We denote by $\mathfrak{osp}(2m|2n)$  the Lie superalgebra produced in the same way but using the superspace $\mC^{2m|2n}$ obtained by omitting the basis vector $v_0$ instead of $V$, and then using the matrices obtained by deleting the first row and column in the description of $\mathfrak{osp}(2m+1|2n)$.
\end{definition}

Definition \ref{osp} is in the odd case consistent with the previous definition of  $\mathfrak{osp}(V)$ if we take the bilinear form $\left\langle -,- \right\rangle$ given by the following block matrix where $\mathbbm{1}_k$ denotes the identity matrix of size $k\times k$:
$$
J= \left(\begin{array}{ccc|cc}
1 & 0 & 0 & 0 & 0\\
0 & 0 & \mathbbm{1}_m & 0 & 0\\
0 & \mathbbm{1}_m & 0 & 0 & 0\\
\hline
0 & 0 & 0 & 0 & -\mathbbm{1}_n \\
0 & 0 & 0 & \mathbbm{1}_n & 0 \\
\end{array}\right),
$$
In the even case one has to delete the first row and column. 

To relate the action of $\mathfrak{osp}(2m+1|2n)$ on $V^{\otimes d}$ to the oriented Brauer category we will restrict the action to the general linear Lie supersubalgebra $\mathfrak{gl}(m|n)$. The embedding $\iota: \mathfrak{gl}(m|n) \hookrightarrow \mathfrak{osp}(2m+1|2n)$ is obtained by only allowing non-zero entries in the matrices $A$, $B$, $C$, and $D$ in the notations above; similarly in the even case. These are all embeddings as Lie super subalgebras obtained by omitting one vertex in the Dynkin diagram (highlighted by a dotted frame).  With respect to the distinguished root system for the Cartan subalgebra consisting of diagonal matrices, see \cite[A.2.1]{Z}, this is given as follows:
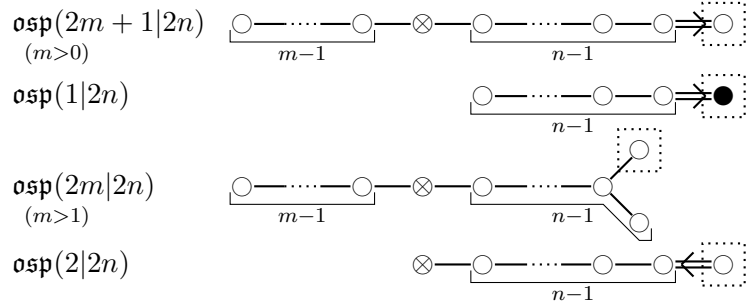
\begin{equation*}
\begin{tikzpicture}[scale = 0.8, thick]
\node [anchor=west] at (-4,0) {$\mathfrak{osp}(2m+1|2n)$};
\node [anchor=west] at (-3.8,-.5) {$\scriptstyle (m > 0)$};

\node at (0,0) {$\ocircle$};
\draw (0.2,0) to +(.5,0);
\draw [dotted] (0.7,0) to +(.6,0);
\draw (1.3,0) to +(.5,0);
\node at (2,0) {$\ocircle$};
\draw (2.2,0) to +(.6,0);
\node at (3,0) {$\otimes$};
\draw (3.2,0) to +(.6,0);
\node at (4,0) {$\ocircle$};
\draw (4.2,0) to +(.5,0);
\draw [dotted] (4.7,0) to +(.6,0);
\draw (5.3,0) to +(.5,0);
\node at (6,0) {$\ocircle$};
\draw (6.2,0) to +(.6,0);
\node at (7,0) {$\ocircle$};
\draw (7.2,.05) to +(.6,0);
\draw (7.7,0) to +(-.2,.2);
\draw (7.7,0) to +(-.2,-.2);
\draw (7.2,-.05) to +(.6,0);
\draw [dotted] (7.65,.35) rectangle +(.7,-.7);
\node at (8,0) {$\ocircle$};
\draw [thin] (-.2,-.1) to +(0,-.2) to +(2.4,-.2) to +(2.4,0);
\node at (1,-.5) {$\scriptstyle m-1$};
\draw [thin] (3.8,-.1) to +(0,-.2) to +(3.4,-.2) to +(3.4,0);
\node at (5.5,-.5) {$\scriptstyle n-1$};

\begin{scope}[yshift=-1.2cm]
\node [anchor=west] at (-4,0) {$\mathfrak{osp}(1|2n)$};

\node at (4,0) {$\ocircle$};
\draw (4.2,0) to +(.5,0);
\draw [dotted] (4.7,0) to +(.6,0);
\draw (5.3,0) to +(.5,0);
\node at (6,0) {$\ocircle$};
\draw (6.2,0) to +(.6,0);
\node at (7,0) {$\ocircle$};
\draw (7.2,.05) to +(.6,0);
\draw (7.7,0) to +(-.2,.2);
\draw (7.7,0) to +(-.2,-.2);
\draw (7.2,-.05) to +(.6,0);
\draw [dotted] (7.65,.35) rectangle +(.7,-.7);
\node at (8,0) {$\newmoon$};
\draw [thin] (3.8,-.1) to +(0,-.2) to +(3.4,-.2) to +(3.4,0);
\node at (5.5,-.5) {$\scriptstyle n-1$};
\end{scope}

\begin{scope}[yshift=-2.7cm]
\node [anchor=west] at (-4,0) {$\mathfrak{osp}(2m|2n)$};
\node [anchor=west] at (-3.8,-.5) {$\scriptstyle (m > 1)$};

\node at (0,0) {$\ocircle$};
\draw (0.2,0) to +(.5,0);
\draw [dotted] (0.7,0) to +(.6,0);
\draw (1.3,0) to +(.5,0);
\node at (2,0) {$\ocircle$};
\draw (2.2,0) to +(.6,0);
\node at (3,0) {$\otimes$};
\draw (3.2,0) to +(.6,0);
\node at (4,0) {$\ocircle$};
\draw (4.2,0) to +(.5,0);
\draw [dotted] (4.7,0) to +(.6,0);
\draw (5.3,0) to +(.5,0);
\node at (6,0) {$\ocircle$};
\draw (6.12,.12) to +(.36,.36);
\draw (6.12,-.12) to +(.36,-.36);
\node at (6.6,.6) {$\ocircle$};
\node at (6.6,-.6) {$\ocircle$};
\draw [dotted] (6.25,.95) rectangle +(.7,-.7);
\draw [thin] (-.2,-.1) to +(0,-.2) to +(2.4,-.2) to +(2.4,0);
\node at (1,-.5) {$\scriptstyle m-1$};
\draw [thin] (3.8,-.1) to +(0,-.2) to +(2.2,-.2) to +(2.8,-.8) to +(3.0,-.8) to +(3.0,-.6);
\node at (5.5,-.5) {$\scriptstyle n-1$};
\end{scope}

\begin{scope}[yshift=-4cm]
\node [anchor=west] at (-4,0) {$\mathfrak{osp}(2|2n)$};

\node at (3,0) {$\otimes$};
\draw (3.2,0) to +(.6,0);
\node at (4,0) {$\ocircle$};
\draw (4.2,0) to +(.5,0);
\draw [dotted] (4.7,0) to +(.6,0);
\draw (5.3,0) to +(.5,0);
\node at (6,0) {$\ocircle$};
\draw (6.2,0) to +(.6,0);
\node at (7,0) {$\ocircle$};
\draw (7.2,.05) to +(.6,0);
\draw (7.3,0) to +(.2,.2);
\draw (7.3,0) to +(.2,-.2);
\draw (7.2,-.05) to +(.6,0);
\draw [dotted] (7.65,.35) rectangle +(.7,-.7);
\node at (8,0) {$\ocircle$};
\draw [thin] (3.8,-.1) to +(0,-.2) to +(3.4,-.2) to +(3.4,0);
\node at (5.5,-.5) {$\scriptstyle n-1$};
\end{scope}
\end{tikzpicture}
\end{equation*}
In the (omitted) classical cases of the orthogonal and symplectic Lie algebra the embedding is the obvious one.

Then the restriction of $V$ to the Lie super subalgebra is given by the followings lemma.

\begin{lemma}
As a $\mathfrak{gl}(m|n)$ module $V=W_\up \oplus W_\down \oplus W_\circ$ in the odd case and $V=W_\up \oplus W_\down$ in the even case, with
$$W_\up = {\rm span}\{v_i \mid i \in I^\up\}, \, W_\down = {\rm span} \{v_i \mid i \in I^\down\} \text{, and } W_\circ = {\rm span} \{v_0\}.$$
Moreover, $W_\up$ is isomorphic to the natural representation of $\mathfrak{gl}(m|n)$, $W_\down$ to its dual and $W_\circ$ to the trivial representation.
\end{lemma}
\begin{proof}
Clearly, $W_\up$ is isomorphic to the natural representation of $\mathfrak{gl}(m|n)$. The dual of the natural representation can be described in an appropriate basis $w_i$ by $ E_{i,j} w_k = - \delta_{i,k} (-1)^{(|i|+|j|)|i|} w_j,$
and so it is isomorphic to $W_\down$ via  $w_i\mapsto v_{\overline{i}}.$ The statement for $W_\circ$ is evident.
\end{proof}

More generally we obtain the following as a direct consequence:
\begin{corollary}
The space $V^{\otimes d}$ decomposes as a $\mathfrak{gl}(m|n)$-module as
\begin{equation} \label{glmn_decomposition}
V^{\otimes d} \cong \bigoplus_{\mathbf{s} \in \widehat{\rm OSeq}[d]} W_{\mathbf{s}_1} \otimes \ldots \otimes W_{\mathbf{s}_d}.
\end{equation}
\end{corollary}

We abbreviate $W_\mathbf{s}=W_{\mathbf{s}_1} \otimes \ldots \otimes W_{\mathbf{s}_d}$ and define for each $\mathbf{s} \in \widehat{\rm OSeq}[d]$ in the odd case (respectively each $\mathbf{s} \in{\rm OSeq}[d]$  in the even case) the set of $\mathbf{s}$-indices $${\rm Vect}(\mathbf{s}) = \{ \underline{i} = (i_1,\ldots,i_d) \mid i_j \in I^{\mathbf{s}_j} \}.$$
Obviously this set labels the standard basis vectors in the summand corresponding to $\mathbf{s}$ in the decomposition (\ref{glmn_decomposition}) with $v_{\underline{i}} = v_{i_1} \otimes \ldots \otimes v_{i_d}$. 

\begin{lemma} \label{lem:iso_modules}
Let $\mathbf{s},\mathbf{t} \in \widehat{\rm OSeq}[d]$ or $\mathbf{s},\mathbf{t} \in{\rm OSeq}[d]$ be two sequences, such that there exists $\sigma \in S_d$ with $\mathbf{s}_{i} = \mathbf{t}_{\sigma (i)}$. Then $\sigma$ induces an isomorphism
\begin{eqnarray}
\label{Rmatrix}
\psi_\sigma:\quad W_{\mathbf{s}} \longrightarrow W_{\mathbf{t}},
\end{eqnarray}
of $\mathfrak{gl}(m|n)$-modules via the action of $S_d$ given in Remark \ref{rem:Sd_action}. Choosing $\sigma$ to be of minimal length defines a distinguished isomorphism $\psi_\sigma$.
\end{lemma}
\begin{proof}
The action of the symmetric group on $V^{\otimes d}$ even commutes with $\mathfrak{gl}(V)$ and hence gives a morphism, and then clearly an isomorphism, of $\mathfrak{gl}(m|n)$-modules when restricted to the subspaces.
\end{proof}
As a consequence we have canonical isomorphisms of algebras $$\op{End}_{\mathfrak{gl}(m|n)}(W_{\mathbf{s}})\cong \op{End}_{\mathfrak{gl}(m|n)}(W_{\mathbf{t}})$$ for each pair ${\mathbf{s}}$, ${\mathbf{t}}$ as above, namely  given by conjugation with  $\psi_\sigma$, where $\sigma$ is the unique choice of minimal length. Note that these endomorphism rings were described in detail in \cite{BS5}. In case $V$ has large enough dimension (in the sense that $d<(m+1)(n+1)$) and $\mathbf{t}=(\up,\ldots,\up,\down,\ldots,\down,\circ,\ldots,\circ)$, where the symbols $\up$ and $\down$ appear exactly $r$ respectively $s$ times, this endomorphism ring is exactly the \emph{walled Brauer algebra} from \cite[Theorem 7.8]{BS5} originally introduced \cite{Tu},  \cite{Kosuda} and studied in \cite{Ni}. 
\begin{lemma} \label{lem:glmnhom_trivial}
Let $\mathbf{s},\mathbf{t} \in \widehat{\rm OSeq}[d]$ and $\ell = | \{i \mid \mathbf{s}_i = \circ\}| - |\{i \mid \mathbf{t}_i = \circ\}|$. Then ${\rm Hom}_{\mathfrak{gl}(m\mid n)}(W_{\mathbf{s}},W_{\mathbf{t}})=\{0\}$ if $\ell$ is odd or $\ell$ is even and then additionally 
\begin{eqnarray}
\label{zero}
  \frac{\ell}{2} &\neq& |\{i \mid \mathbf{s}_i = \up\}| - |\{i \mid \mathbf{t}_i = \up\}| .
  \end{eqnarray}
\end{lemma}
\begin{proof}
Let ${\rm ht}(\lambda)$ denote the height of a $\mathfrak{gl}(m|n)$
weight, i.e. the sum of all coefficients written with respect to the
standard $\{ \varepsilon_i \}$-basis. In this basis the weights occurring
in $W_\up$ are $\varepsilon_1, \ldots, \varepsilon_{m+n}$, those in
$W_{\down}$ are $-\varepsilon_1, \ldots, -\varepsilon_{m+n}$, and in
$W_{\circ}$ only the weight $0$ occurs. If $\ell$ is odd, then the
height of the weights occurring in $W_\mathbf{s}$ will either all be even or all be odd, while in $W_\mathbf{t}$ it will be the other way around, thus
there can't be a non-trivial $\mathfrak{gl}(m|n)$-morphism. 

If $\ell$ is even and \eqref{zero} holds define
$\mathbf{s}_\up = |\{i \mid \mathbf{s}_i = \up\}|$ and analogously
$\mathbf{s}_\down$, $\mathbf{t}_\up$, and $\mathbf{t}_\down$. Assume now
$\mathbf{s}_\up - \mathbf{t}_\up > \ell/2$, then it follows that
$\mathbf{s}_\down - \mathbf{t}_\down < \ell/2$. By adjointness and
Lemma \ref{lem:iso_modules} we know that
$${\rm Hom}_{\mathfrak{gl}(m\mid n)}(W_{\mathbf{s}},W_{\mathbf{t}})\cong {\rm
Hom}_{\mathfrak{gl}(m\mid n)}(W_\up^{\otimes (\mathbf{s}_\up +
\mathbf{t}_\down)},W_\up^{\otimes (\mathbf{t}_\up + \mathbf{s}_\down)}).$$
Then $\mathbf{s}_\up + \mathbf{t}_\down < \mathbf{t}_\up + \mathbf{t}_\down
- \ell/2 < \mathbf{t}_\up + \mathbf{s}_\down$, and so the height of any
weight appearing in the domain is strictly less than the height of those appearing in
the codomain and so all $\mathfrak{gl}(m|n)$ morphism are trivial. The
remaining case $\mathbf{s}_\up - \mathbf{t}_\up < \ell/2$ is done in exactly the
same way.
\end{proof}

To extend the action from Theorem \ref{thm:Brauer_action} to the oriented Brauer category we need the notion of a weight of an oriented Brauer diagram:

\begin{definition}
Let $\mathbf{s},\mathbf{t} \in \widehat{\rm OSeq}$. Assume $b \in \widehat{\mathcal{B}}[d]_{\mathbf{s}}^{\mathbf{t}}$ and $\underline{i} \in {\rm Vect}(\mathbf{s})$ and $\underline{j} \in {\rm Vect}(\mathbf{t})$. We denote by $b^{\underline{j}}_{\underline{i}}$ the diagram $b$ with the vertices at the bottom labelled by the elements $i_1,\ldots,i_d$ and the vertices on the top by $j_1,\ldots,j_d$ (read always from left to right). We say that $b^{\underline{j}}_{\underline{i}}$ is \emph{consistently labelled} if the labelling function $\mid \mid \op{L}(-) \mid \mid$ defined by $\op{L}(k)=j_k$ and $\op{L}(k^*)=i_k$, for $1 \leq k \leq d$,  is constant on the subsets of the partition $b$. 
\end{definition}
\begin{example}
\label{labelled}
The first of the following diagrams is consistently labelled,  the second and third are not (the highlighted parts are violating the conditions): 
\begin{equation}
\label{Ex21}
\begin{tikzpicture}[scale=0.5,thick,>=angle 90]
\draw [>->] (0,2) node[above] {$\overline{1}$} to [out=-90,in=-180]
+(.5,-.5) to [out=0,in=-90] +(.5,.5) node[above] {$1$};
\draw [<-<] (0,0) node[below] {$\overline{2}$} to [out=90,in=-180] +(1,1)
to [out=0,in=-90] +(1,1) node[above] {$\overline{2}$};
\draw [>->] (2,0) node[below] {$3$} to [out=90,in=-180] +(.5,.5) to
[out=0,in=90] +(.5,-.5) node[below] {$\overline{3}$} ;
\draw (1,0) node[below] {$0$} circle(2.5pt);
\draw (3,2) node[above] {$0$} circle(2.5pt);

\begin{scope}[xshift=5cm]
\draw [dotted] (-.3,2.9) rectangle +(1.6,-1.6);
\draw [>->] (0,2) node[above] {$1$} to [out=-90,in=-180] +(.5,-.5) to
[out=0,in=-90] +(.5,.5) node[above] {$1$};
\draw [<-<] (0,0) node[below] {$\overline{2}$} to [out=90,in=-180] +(1,1)
to [out=0,in=-90] +(1,1) node[above] {$\overline{3}$};
\draw [>->] (2,0) node[below] {$3$} to [out=90,in=-180] +(.5,.5) to
[out=0,in=90] +(.5,-.5) node[below] {$\overline{3}$} ;
\draw (1,0) node[below] {$0$} circle(2.5pt);
\draw (3,2) node[above] {$0$} circle(2.5pt);
\end{scope}

\begin{scope}[xshift=10cm]
\draw [dotted] (-.3,2.9) rectangle +(1.6,-1.6);
\draw [>->] (0,2) node[above] {$1$} to [out=-90,in=-180] +(.5,-.5) to
[out=0,in=-90] +(.5,.5) node[above] {$1$};
\draw [<-<] (0,0) node[below] {$\overline{2}$} to [out=90,in=-180] +(1,1)
to [out=0,in=-90] +(1,1) node[above] {$\overline{2}$};
\draw [>->] (2,0) node[below] {$3$} to [out=90,in=-180] +(.5,.5) to
[out=0,in=90] +(.5,-.5) node[below] {$\overline{3}$} ;

\draw [dotted] (.7,.3) rectangle +(.6,-1.3);
\draw (1,0) node[below] {$2$} circle(2.5pt);
\draw (3,2) node[above] {$0$} circle(2.5pt);
\end{scope}
\end{tikzpicture}
\end{equation}
\end{example}

\begin{definition}
The \emph{weight} ${\rm wt}\left(b^{\underline{j}}_{\underline{i}}\right)$ of a labelled oriented diagram $b^{\underline{j}}_{\underline{i}}$ is defined as follows:
$${\rm wt}\left(b^{\underline{j}}_{\underline{i}}\right) := \left\lbrace
\begin{array}{ll}
 \prod_{c} (-1)^{\mid c \mid} \prod_{h}(-1) \prod_{h'} (-1),& \text{if }{}b^{\underline{j}}_{\underline{i}} \text{ is consistently labelled,}\\
0, & \text{otherwise,}
\end{array}\right.,$$
where, viewed as diagrams,
\begin{enumerate}
\item the first product is over all crossings $c$ of two different strands with the notation $| c | = |i| \cdot |j|$ if the two strands are labelled by $i$ and $j$;
\item the second product is over all labelled clockwise caps $h$ with large labels, i.e. oriented horizontal strands at the bottom of the diagram with left endpoint oriented $\up$ and with labelling set $\{a, \overline{a}\}$ such that $\mid\mid a\mid\mid>m$;
\item the third product is over all labelled anticlockwise cups $h'$ with large labels, i.e. oriented horizontal strands at the top of the diagram with left endpoint oriented $\down$ and with labelling set $\{a, \overline{a}\}$ such that $\mid\mid a\mid\mid>m$.
\end{enumerate}
\end{definition}
\begin{remark}
\label{rotatedcross}
\rm{Note that the weight of a diagram depends on $V$. For instance the first diagram in \eqref{Ex21} has weight $0$ if $m\geq 3$, has weight $-1$ if $m=1,2$ and has weight $1$ if $m=0$.  Note that the weight of a crossing is by definition independent of the four possible orientations:
\begin{equation}
\begin{tikzpicture}[scale=0.7,thick,>=angle 90]
\draw (-.8,.75) node {$t_1=$};
\draw [<-<] (0,0) to [out=90,in=-135] +(.5,.75) to [out=45,in=-90] +(.5,.75);
\draw [<-<] (1,0) to [out=90,in=-45] +(-.5,.75) to [out=135,in=-90]
+(-.5,.75);

\begin{scope}[xshift=3cm]
\draw (-.8,.75) node {$t_2=$};
\draw [<-<] (0,0) to [out=90,in=-135] +(.5,.75) to [out=45,in=-90] +(.5,.75);
\draw [>->] (1,0) to [out=90,in=-45] +(-.5,.75) to [out=135,in=-90]
+(-.5,.75);
\end{scope}

\begin{scope}[xshift=6cm]
\draw (-.8,.75) node {$t_3=$};
\draw [>->] (0,0) to [out=90,in=-135] +(.5,.75) to [out=45,in=-90] +(.5,.75);
\draw [<-<] (1,0) to [out=90,in=-45] +(-.5,.75) to [out=135,in=-90]
+(-.5,.75);
\end{scope}

\begin{scope}[xshift=9cm]
\draw (-.8,.75) node {$t_4=$};
\draw [>->] (0,0) to [out=90,in=-135] +(.5,.75) to [out=45,in=-90] +(.5,.75);
\draw [>->] (1,0) to [out=90,in=-45] +(-.5,.75) to [out=135,in=-90]
+(-.5,.75);
\end{scope}
\end{tikzpicture}
\end{equation}
On the other hand, we can ``rotate'' each crossing by applying a cup and a cap on top of the diagram and compute then its weight. But in any case the weights of the added cup and cap multiply to $1$, and so the weights don't change. In fact, we will see in the proof of Theorem \ref{firstmain} that the weight is an invariant on equivalence classes of oriented labelled diagrams.}
\end{remark}
By a \emph{representation} of $\mathcal{OB}_d(m-n)$ we mean a linear functor $F$ from $\mathcal{OB}_d(m-n)$ to the category of complex vector spaces. Equivalently we say $\mathcal{OB}_d(m-n)$  acts on $X:=\bigoplus_{\widehat{\rm OSeq}[d]}F(\mathbf{s})$. 
If $X$ is moreover an $R$-module for some ring $R$, then the action of $\mathcal{OB}_d(m-n)$  \emph{commutes with the action of $R$} if $F(f)(r.x)=r.F(f)(x)$ for any $x\in X, r\in R$ and morphism $f$ in $\mathcal{OB}_d(m-n)$.   

\begin{definition}
For $\mathbf{s},\mathbf{t} \in \widehat{\rm OSeq}$ and $b \in \widehat{\mathcal{B}}_{\mathbf{s}}^{\mathbf{t}}$ we define the linear map
\begin{eqnarray}
F(b):\quad W_\mathbf{s}&\longrightarrow&W_{\mathbf{t}}\nonumber \text{, via } F(b)(v_{\underline{i}}) := \sum_{\underline{j} \in {\rm Vect}(\mathbf{t})} {\rm wt}\left(b^{\underline{j}}_{\underline{i}}\right) v_{\underline{j}}.\label{Ff}
\end{eqnarray}
on basis vectors $v_{\underline{i}}$, where  $\underline{i} \in {\rm Vect}(\mathbf{s})$.
\end{definition}

\begin{theorem}
\label{firstmain}
Let $\delta$ be the supertrace of $V$. The assignment
\begin{eqnarray*}
\mathbf{s}\mapsto F(\mathbf{s}):=W_{\mathbf{s}}&\text{and}&f\mapsto F(f),
\end{eqnarray*}
on objects respectively morphisms defines a functor from $\mathcal{OB}(m-n)$ to the category of finite dimensional $\mathfrak{gl}(m|n)$-modules. 

It restricts to an action of the oriented Brauer category $\mathcal{OB}_d(m-n)$ on $V^{\otimes d}$ which commutes with the action of $\mathfrak{gl}(m|n)$. The identity of the object $\mathbf{s}\in \widehat{\rm OSeq}[d]$ acts by projecting onto the summand $W_{\mathbf{s}}$. 
\end{theorem}

\begin{proof} 
By definition, the identity of the object $\mathbf{s}$ acts by projecting onto the summand $W_{\mathbf{s}}$. To see that the action is well-defined note that two diagrams $D,D\in\widehat{\mathcal{B}}_{\mathbf{s}}^{\mathbf{t}}$ are the same if and only if their reductions $\overline{D}, \overline{D'}$ are the same. Moreover,  $F(D)=F(D')$ if and only if $F(\overline{D})=F(\overline{D'})$.  Finally assume $D_1\in'\widehat{\mathcal{B}}_{\mathbf{s}}^{\mathbf{t}}, D_2\in'
\widehat{\mathcal{B}}_{\mathbf{t}}^{\mathbf{u}}$, and $D'\in\widehat{\mathcal{B}}_{\mathbf{s}}^{\mathbf{u}}$ then $F(D_2\cdot D_1)=F(D')$ if  $F(\overline{D_2}\cdot \overline{D_1})=F(\overline{D'})$
Hence, to see that the action is well-defined and compatible with multiplication it is enough to restrict to reduced sequences. By Remark \ref{Jon} it suffices then to show that the oriented Brauer category in the sense of \cite{Betal} acts. Luckily in this case we have explicit (monoidal) generators and relations for the morphisms, namely the generators

\begin{equation}
\begin{tikzpicture}[scale=0.7,thick,>=angle 90]
\draw (-.8,.5) node {$c=$};
\draw [-] (0,1) to [out=-90,in=-180] +(.5,-.5);
\draw [<-] (.5,.5) to [out=0,in=-90] +(.5,.5);

\begin{scope}[xshift=3cm]
\draw (-.8,.5) node {$d=$};
\draw [-] (0,0) to [out=90,in=-180] +(.5,.5);
\draw [<-] (.5,.5) to [out=0,in=90] +(.5,-.5);
\end{scope}

\begin{scope}[xshift=6cm]
\draw (-.8,.5) node {$s=$};
\draw [->] (0,0) to [out=90,in=-135] +(.5,.75) to [out=45,in=-90] +(.5,.75);
\draw [->] (1,0) to [out=90,in=-45] +(-.5,.75) to [out=135,in=-90]
+(-.5,.75);
\end{scope}
\end{tikzpicture}
\end{equation}
with all possible labellings and the relations (1.4)-(1.9) in \cite{Betal}. We only need to keep track of the weights: The first two of these relations amount to the fact that a consistently oriented and labelled kink built from a cup and cap has weight $1$, the third just requires that $(-1)^{|i||j|}(-1)^{|i||j|}=1$ for any labels $i,j$, the fourth just requires $$(-1)^{|i||j|}(-1)^{|i||k|}(-1)^{|j||k|}= (-1)^{|j||k|}(-1)^{|i||k|}(-1)^{|i||j|}$$ and the relation \cite[(1.8)]{Betal} amounts to the fact that the weight of a crossing does not depend on the orientation, but only on the labels, see Remark \ref{rotatedcross}. Finally for \cite[(1.9)]{Betal} it is enough to see that the sum of the weights of the diagram $d\cdot t_3\cdot c$ over all labellings equals the supertrace $\delta$. But the labellings contributing weight $1$ are precisely those with absolute value at most $m$, hence there are precisely $m$. The other $n$ possible labelling always create a cup of weight one and a cap of weight $-1$ or vice versa. Hence, the total weight is $\delta=m-n$. Hence we obtain an action when restricted to reduced sequences and therefore the claim follows.
\end{proof}

\section{The Isomorphism Theorem}

In this section we prove Theorem~ \ref{thm:iso} from the introduction. The main step is to establish a commuting diagram of the form
\begin{eqnarray}
\begin{xy}
  \xymatrix{
  {\rm End}_{\mathfrak{osp}(m'|2n)}\left(V^{\otimes d} \right)
\ar@{^{(}->}[rr]& &{\rm End}_{\mathfrak{gl}(m|n)}\left(
\bigoplus_{\mathbf{s}} W_\mathbf{s}\right) \\
  {\rm Br}_d(\delta) \ar@{->}[rr]^{\Phi}  \ar@{-->}[u]^{\Psi}
\ar[urr]& &{\rm End}_{{\rm Mat}(\mathcal{OB}_d(m-n))}\left(
\bigoplus_{\mathbf{s}} \mathbf{s}\right) \ar[u]_{\Theta}
  }
\end{xy}
\label{THEdiagram}
\end{eqnarray}
where $\delta$ is the supertrace of $V$, $m'=2m+1$ or $m'=2m$ and $\Phi$ a map constructed from the action of $\mathcal{OB}(m-n)$ from Theorem \ref{firstmain}. For this we view the algebra ${\rm Br}_d(\delta)$ as a category with one object, called $\star$, and endomorphism ring ${\rm Br}_d(\delta)$, and construct a functor into ${\rm Mat}(\mathcal{OB}(m-n))$, the additive closure of $\mathcal{OB}(m-n)$.

\begin{definition}
The category ${\rm Mat}(\mathcal{OB}(m-n))$ is defined as follows: The objects of ${\rm Mat}(\mathcal{OB}(m-n))$ are formal finite direct sums of objects in $\mathcal{OB}(m-n)$, and the morphisms are matrices of morphisms between the summands with addition and composition given by the usual rules of matrix multiplication, see e.g. \cite{BarNatan} for more details.
\end{definition}

The idea is now to construct a functor which sends the one object of ${\rm Br}_d(\delta)$ to the direct sum of all $\mathbf{s} \in \widehat{\rm OSeq}[d]$ and a Brauer diagram $b$ to a matrix with rows and columns indexed by $\widehat{\rm OSeq}[d]$. It will be convenient to write a matrix $A$ as $\sum_{\mathbf{s},\mathbf{t}}1_{\mathbf{t}} A 1_{\mathbf{s}}$. Diagrammatically we will write the morphisms as formal sums of oriented diagrams, where the orientation sequences stand for the matrix idempotents, see e.g. \eqref{Psisi}.

Due to the fact that we also have the "place holder" symbols $\circ$ in the sequences we need the following additional notion: 

\begin{definition}
Let $b,b' \in \widehat{\mathcal{B}}[d]$ be generalized Brauer diagrams.Then $b' \in \widehat{\mathcal{B}}[d]$ is called a  \emph{subdiagram} of $b \in \mathcal{B}[d]$, denoted $b' \lhd b$, if $b'$ refines the partition of $b$. In other words, the diagram $b'$ is obtained from $b$ by removing some arcs with their orientations and replace them with the appropriate number of $\circ$'s. 
\end{definition}


This allows us to define the functor from ${\rm Br}_d(\delta)$ to ${\rm Mat}(\mathcal{OB}(m-n))$.
\begin{proposition} \label{prop:Brauer_embedded}
The assignment 
\begin{eqnarray*}
\Psi:\quad\star \mapsto \bigoplus_{\mathbf{s} \in \widehat{\rm OSeq}[d]} \mathbf{s} \text{ and } b \mapsto M(b),
\end{eqnarray*}
with $M(b)_{(\mathbf{t},\mathbf{s})} \in Hom_{\mathcal{OB}(m-n)}(\mathbf{s},\mathbf{t})$ equal to the unique subdiagram of $b$ in $\widehat{\mathcal{B}}[d]_\mathbf{s}^\mathbf{t}$ or zero if such a diagram does not exist, defines a faithful functor from ${\rm Br}_d(\delta)$ to ${\rm Mat}(\mathcal{OB}(m-n))$.
\end{proposition}
\begin{proof}
Clearly, the required subdiagram of $b$ is unique in case it exists. That the functor is faithful is obvious by definition. To see that the functor is well-defined it remains to verify the relations \eqref{relBrauer}. The element $s_i$ is sent to the sum of all oriented crossings at place $i,i+1$ plus its subdiagrams obtained by removing one or two strands in the crossing, for instance for $d=2$:
\begin{equation}
\label{Psisi}.
\begin{tikzpicture}[scale=0.6,thick,>=angle 90]
\draw (-1.3,.75) node {$M(s_i) = $};
\draw [>->] (0,0) to [out=90,in=-135] +(.5,.75) to [out=45,in=-90] +(.5,.75);
\draw [>->] (1,0) to [out=90,in=-45] +(-.5,.75) to [out=135,in=-90]
+(-.5,.75);
\draw (1.5,.75) node {$+$};

\begin{scope}[xshift=2cm]
\draw [<-<] (0,0) to [out=90,in=-135] +(.5,.75) to [out=45,in=-90] +(.5,.75);
\draw [>->] (1,0) to [out=90,in=-45] +(-.5,.75) to [out=135,in=-90]
+(-.5,.75);
\draw (1.5,.75) node {$+$};
\end{scope}

\begin{scope}[xshift=4cm]
\draw [>->] (0,0) to [out=90,in=-135] +(.5,.75) to [out=45,in=-90] +(.5,.75);
\draw [<-<] (1,0) to [out=90,in=-45] +(-.5,.75) to [out=135,in=-90]
+(-.5,.75);
\draw (1.5,.75) node {$+$};
\end{scope}

\begin{scope}[xshift=6cm]
\draw [<-<] (0,0) to [out=90,in=-135] +(.5,.75) to [out=45,in=-90] +(.5,.75);
\draw [<-<] (1,0) to [out=90,in=-45] +(-.5,.75) to [out=135,in=-90]
+(-.5,.75);
\draw (1.5,.75) node {$+$};
\end{scope}

\begin{scope}[xshift=8cm]
\draw [>->] (0,0) to [out=90,in=-135] +(.5,.75) to [out=45,in=-90] +(.5,.75);
\draw (1,0) circle(2.5pt);
\draw (0,1.5) circle(2.5pt);
\draw (1.5,.75) node {$+$};
\end{scope}

\begin{scope}[xshift=10cm]
\draw [<-<] (0,0) to [out=90,in=-135] +(.5,.75) to [out=45,in=-90] +(.5,.75);
\draw (1,0) circle(2.5pt);
\draw (0,1.5) circle(2.5pt);
\draw (1.5,.75) node {$+$};
\end{scope}

\begin{scope}[xshift=12cm]
\draw [>->] (1,0) to [out=90,in=-45] +(-.5,.75) to [out=135,in=-90]
+(-.5,.75);
\draw (0,0) circle(2.5pt);
\draw (1,1.5) circle(2.5pt);
\draw (1.5,.75) node {$+$};
\end{scope}

\begin{scope}[xshift=14cm]
\draw [<-<] (1,0) to [out=90,in=-45] +(-.5,.75) to [out=135,in=-90]
+(-.5,.75);
\draw (0,0) circle(2.5pt);
\draw (1,1.5) circle(2.5pt);
\end{scope}
\end{tikzpicture}
\end{equation}
If $d>2$, then the assignment looks locally as above with the remaining strands oriented in all possible ways.
%
%
%

Note that at vertex $i$ and $i+1$ (resp. $i^*$ and $(i+1)^*$)  every possible combination from $\{\up,\down,\circ\}$ occurs precisely once and composing $\Psi(s_i)$ with itself gives the identity. The other two braid relations can also be checked easily.  The image of $e_i$ looks locally as follows: 
\begin{equation*}
\label{Psiei}
\begin{tikzpicture}[scale=0.6,thick,>=angle 90]
\draw (-1.3,.75) node {$M(e_i) = $};
\draw [>->] (0,0) to [out=90,in=-180] +(.5,.5) to [out=0,in=90] +(.5,-.5);
\draw [>->] (0,1.5) to [out=-90,in=-180] +(.5,-.5) to [out=0,in=-90]
+(.5,.5);
\draw (1.5,.75) node {$+$};

\begin{scope}[xshift=2cm]
\draw [<-<] (0,0) to [out=90,in=-180] +(.5,.5) to [out=0,in=90] +(.5,-.5);
\draw [>->] (0,1.5) to [out=-90,in=-180] +(.5,-.5) to [out=0,in=-90]
+(.5,.5);
\draw (1.5,.75) node {$+$};
\end{scope}

\begin{scope}[xshift=4cm]
\draw [>->] (0,0) to [out=90,in=-180] +(.5,.5) to [out=0,in=90] +(.5,-.5);
\draw [<-<] (0,1.5) to [out=-90,in=-180] +(.5,-.5) to [out=0,in=-90]
+(.5,.5);
\draw (1.5,.75) node {$+$};
\end{scope}

\begin{scope}[xshift=6cm]
\draw [<-<] (0,0) to [out=90,in=-180] +(.5,.5) to [out=0,in=90] +(.5,-.5);
\draw [<-<] (0,1.5) to [out=-90,in=-180] +(.5,-.5) to [out=0,in=-90]
+(.5,.5);
\draw (1.5,.75) node {$+$};
\end{scope}

\begin{scope}[xshift=8cm]
\draw [>->] (0,0) to [out=90,in=-180] +(.5,.5) to [out=0,in=90] +(.5,-.5);
\draw (0,1.5) circle(2.5pt);
\draw (1,1.5) circle(2.5pt);
\draw (1.5,.75) node {$+$};
\end{scope}

\begin{scope}[xshift=10cm]
\draw [<-<] (0,0) to [out=90,in=-180] +(.5,.5) to [out=0,in=90] +(.5,-.5);
\draw (0,1.5) circle(2.5pt);
\draw (1,1.5) circle(2.5pt);
\draw (1.5,.75) node {$+$};
\end{scope}

\begin{scope}[xshift=12cm]
\draw [>->] (0,1.5) to [out=-90,in=-180] +(.5,-.5) to [out=0,in=-90]
+(.5,.5);
\draw (0,0) circle(2.5pt);
\draw (1,0) circle(2.5pt);
\draw (1.5,.75) node {$+$};
\end{scope}

\begin{scope}[xshift=14cm]
\draw [<-<] (0,1.5) to [out=-90,in=-180] +(.5,-.5) to [out=0,in=-90]
+(.5,.5);
\draw (0,0) circle(2.5pt);
\draw (1,0) circle(2.5pt);
\draw (1.5,.75) node {$+$};
\end{scope}

\begin{scope}[xshift=16cm]
\draw (0,1.5) circle(2.5pt);
\draw (1,1.5) circle(2.5pt);
\draw (0,0) circle(2.5pt);
\draw (1,0) circle(2.5pt);
\end{scope}
\end{tikzpicture}
\end{equation*}
If we compose it with itself then we obtain the same sum, but each oriented diagram with an additional clockwise circle, an additional anticlockwise circle and in the odd case also the original diagram itself. Hence, we obtain $\delta$ times the original diagram as required.  The fifth relation in \eqref{relBrauer} is clear. For the next relation note that $\Psi(e_i)\Psi(e_{i+1})\Psi(e_i)$  is a sum of oriented diagrams of the same underlying shape as the one for $e_ie_{i+1}e_i$ or subdiagrams of this form, but equipped with  orientations. Note that each fixed pair of orientation at the bottom and the top appears in precisely one summand and so $\Psi$ preserves the relation $e_ie_{i+1}e_i=e_i$ and similarly also $e_{i+1}e_ie_{i+1}=e_{i+1}$. One can easily check that it also preserves $e_is_i=e_i=s_ie_i$ and the last two relations of \eqref{relBrauer}.
\end{proof}

Thus, we have now two actions of the Brauer algebra on $V^{\otimes d}$, one by Theorem \ref{thm:Brauer_action} and another one given by Proposition \ref{prop:Brauer_embedded} and Theorem \ref{firstmain}.

\begin{lemma} \label{lem:actions_agree}
The actions of the Brauer algebra given in Theorem \ref{thm:Brauer_action} and the one given by Theorem \ref{firstmain} via the embedding of Proposition \ref{prop:Brauer_embedded} agree.
\end{lemma}
\begin{proof}
This follows directly from the definitions by a direct calculation on the generators.
\end{proof}

By Proposition \ref{prop:Brauer_embedded} we have the induced map $\Phi$ from the Brauer algebra to the algebra ${\rm End}_{{\rm Mat}(\mathcal{OB}_d(m-n))}\left( \bigoplus_{\mathbf{s}} \mathbf{s}\right)$. Moreover we have the action map $\Theta$ from ${\rm End}_{{\rm Mat}(\mathcal{OB}_d(m-n))}\left( \bigoplus_{\mathbf{s}} \mathbf{s}\right)$ to ${\rm End}_{\mathfrak{gl}(m|n)}\left(V^{\otimes d} \right)$ and know from Lemma~\ref{lem:actions_agree} that the image of $\Theta \circ \Phi$ is contained in ${\rm End}_{\mathfrak{osp}(2m+1|2n)}\left(V^{\otimes d} \right)$. Hence we get the induced map $\Psi$ as indicated in \eqref{THEdiagram}. We claim that $\Psi$ is an isomorphism if $d \leq m+n$. First note that it is injective, since $\Phi$ is injective by definition, and $\Theta$ is injective if $d<(m+1)(n+1)$ by \cite[Theorem 7.8]{BS5}, in particular it is injective if $d\leq m+n$.

Our strategy to prove surjectivity will be to use the surjectivity of $\Theta$ from \cite[Theorem 7.8]{BS5} and show that any element of ${\rm End}_{{\rm Mat}(\mathcal{OB}_d(m-n))}\left( \bigoplus_{\mathbf{s}} \mathbf{s}\right)$ that commutes with the action of the ortho-symplectic Lie superalgebra is already contained in the image of $\Phi$.

We will show this by an inductive argument. For this we subdivide our set of generalized Brauer diagrams into smaller sets:
$$ \widehat{\mathcal{B}}[d] = \coprod_{1 \leq k \leq d, 1 \leq r \leq d} \widehat{\mathcal{B}}[d]_{(k,r)},$$
where $\widehat{\mathcal{B}}[d]_{(k,r)}$ denotes the set of diagrams with exactly $2r$ singleton subsets and $k$ vertical strands, i.e. subsets of the form $\{i,j^*\}$ for some $i,j$. Furthermore, we note that for the even case we can embed the set of Brauer diagrams into the set of generalized Brauer diagrams as
\begin{eqnarray} \label{eqn:diagram_decomp}
\mathcal{B}[d] = \coprod_{1 \leq k \leq d} \widehat{\mathcal{B}}[d]_{(k,0)}.
\end{eqnarray}
  
\begin{theorem} \label{thm:maintwo}
Assume that one of the following conditions holds:
\begin{eqnarray}
{\rm sdim} V \neq 2m|0&\text{and}&d \leq m+n\text { or}\label{eq1}\\
{\rm sdim} V = 2m|0&\text{with}&m > 0 \text{ and }d < m.\label{eq2}
\end{eqnarray}
Then the map $\Psi$ is an isomorphism, i.e., ${\rm Br}_d(\delta) \cong {\rm End}_{\mathfrak{osp}(V)}\left(V^{\otimes d} \right)$.
\end{theorem}
\begin{proof}
It only remains to prove the surjectivity. For this we proceed as follows: Given an element $f \in {\rm End}_{{\rm Mat}(\mathcal{OB}_d(m-n))}\left( \bigoplus_{\mathbf{s}} \mathbf{s}\right)$ which commutes with the action of $\mathfrak{osp}(V)$ we will show that there is a recursive procedure to write $f$ as a linear combination of elements $M(b)$ for $b \in \mathcal{B}[d]$. This will be done by successively subtracting multiples of $M(b)$'s with decreasing numbers of vertical strands. Since the $M(b)$ for $b \in \mathcal{B}[d]$ span the image of the map $\Psi$, the claim follows. Write
$$ f = \sum_{(\mathbf{t}, b, \mathbf{s}) \text{ oriented}} \gamma_{\mathbf{s},b,\mathbf{t}} 1_\mathbf{s} b 1_\mathbf{t}.$$
Let $k$ be maximal and then $r$ minimal such that there exists an oriented generalized Brauer diagram $(\mathbf{t},b,\mathbf{s})$ with $b \in \widehat{\mathcal{B}}[d]_{(k,r)}$ and $\gamma_{\mathbf{s},b,\mathbf{t}} \neq 0$. 

\begin{claim}[1]
\label{Claim1}
The equality $r=0$ holds. In particular we can find $b \in \widehat{\mathcal{B}}[d]_{(k,0)}$ such that there exists an oriented Brauer diagram $(\mathbf{t},b,\mathbf{s})$ with $\gamma_{\mathbf{s},b,\mathbf{t}} \neq 0$. 
\end{claim}
For the even case this holds by definition and (\ref{eqn:diagram_decomp}). For the odd case this is Corollary~\ref{cor:eliminate_subdiagrams} below. The following is proved in the next paragraph:
\begin{claim}[2]
\label{Claim2}
Let $(\mathbf{t}',b,\mathbf{s}')$ be an oriented Brauer diagram with the same underlying diagram $b$ from Claim (1). Then $\gamma_{\mathbf{s}',b,\mathbf{t}'} = \gamma_{\mathbf{s},b,\mathbf{t}}$.
\end{claim}
We assume for now the two claims hold and fix $b \in \widehat{\mathcal{B}}[d]_{(k,0)}$ as in Claim (1). We denote by $\gamma_b\not=0$ its coefficient in $f$. By Claim (2) this is well-defined, ie. independent of a chosen orientation. Then define
$$ f' = f - \sum_{b \in \widehat{\mathcal{B}}[d]_{(k,0)}} \gamma_b M(b).$$

Since $M(b) \in {\rm im} (\Phi)$ for all $b\in \widehat{\mathcal{B}}[d]_{(k,0)}$ the surjectivity of the theorem follows if we show that $f' \in  {\rm im} (\Phi)$. Thanks to Claim (1) we know that $f'$ is contained in the span of the oriented generalized Brauer diagrams $(\mathbf{r},c,\mathbf{p})$ with $c \in \widehat{\mathcal{B}}[d]_{(l,q)}$ where $l \leq k-1$ or $l=k$ and $q\not=0$. But then by Claim (1) we have $l \leq k-1$ and some $q$. 
Hence, either $f'=0$ or we can repeat our arguments for $f'$ instead of $f$ and our maximal choice of $k$ strictly decreases in each step. Hence after finitely many steps we reduced the question whether $f \in {\rm im} (\Phi)$ to the question whether $0 \in {\rm im} (\Phi)$. This is certainly true, and thus the theorem follows.
\end{proof}
 
\begin{proof}[Proof of Claim (2)]
Our proof treats the situations \eqref{eq1} and \eqref{eq2} separately distinguishing moreover in \eqref{eq1} the cases $n>0$ respectively $n=0$.

Let first  $X$ be the unique element in $\mathfrak{osp}(V)$ which maps $v_{m+1}$ to $v_{\overline{m+1}}$ and annihilates all other basis elements, that is in terms of Definition~\ref{osp} the matrix $d_1$ with exactly one non-zero entry $1$ in the upper left corner. Its transpose $X^T$ maps $v_{\overline{m+1}}$ to $v_{{m+1}}$ and annihilates all other basis vectors.

\textit{Assume \eqref{eq1}, ie. $n > 0$ and $d \leq n+m$:} It is enough to consider the situation where  $(\mathbf{t},b,\mathbf{s})$ and $(\mathbf{t}',b,\mathbf{s}')$ differ only in the orientation of one strand $S$, since otherwise we can repeat the argument. Thanks to the assumption $d\leq n+m$ we can pick $\underline{i} \in {\rm Vect}(\mathbf{s})$, $\underline{j} \in {\rm Vect}(\mathbf{t})$ to get a consistent labelling of $(\mathbf{t},b,\mathbf{s})$ with the following property
\begin{center}
{\it Different strands are labelled with different absolute values, and the strand $S$ is the unique strand labelled with absolute value $m+1$.}
\end{center}

To prove Claim (2) we have to distinguish between three cases, namely the cases where $S$ is a vertical strand, a cup or a cap respectively, ie. the cases where from the two labels $m+1$ and $\overline{m+1}$ exactly one, none or both occur in $\underline{i}$.

For an arbitrary labelling sequence $\underline{i}$ we denote by  $\underline{i^{\uparrow}}$,  $\underline{i}^{\downarrow}$ and  $\underline{i}^{\updownarrow}$ the sequence obtained by changing all $m+1$'s into $\overline{m+1}$'s, by changing all $\overline{m+1}$'s into $m+1$'s, or by swapping the labels $\overline{m+1}$ and  $m+1$ respectively. 

\textbf{Case I: Vertical Strand.} Assume $S$ is labelled $m+1$ (the case of the label $\overline{m+1}$ is done by replacing the role of $X$ with $X^T$).  Since the label $m+1$ occurs only at strand $S$ it follows that 
$Xv_{\underline{i}}= v_{\underline{i}^{\uparrow}}$. Moreover, $\underline{i}^{\uparrow}\in {\rm Vect}(\mathbf{s}')$,  $\underline{j}^{\uparrow}\in {\rm Vect}(\mathbf{t}')$. The equivariance \eqref{oh} of $f$ implies
$$ \left\langle f(Xv_{\underline{i}}),v_{\underline{j}^{\uparrow}}^*\right\rangle = - \left\langle f(v_{\underline{i}}),Xv_{\underline{j}^{\uparrow}}^*\right\rangle= \left\langle f(v_{\underline{i}^{\uparrow}}),v_{\underline{j}^{\uparrow}}^*\right\rangle = \gamma_{\mathbf{s}',b,\mathbf{t}'} {\rm wt}\left(b_{\underline{i}^{\uparrow}}^{\underline{j}^{\uparrow}}\right),$$
 Hence
\begin{equation} \label{eqn:case1_1}
\left\langle f(Xv_{\underline{i}}),v_{\underline{j}}^*\right\rangle = \left\langle f(v_{\underline{i}^{\uparrow}}),v_{\underline{j}^{\uparrow}}^*\right\rangle = \gamma_{\mathbf{s}',b,\mathbf{t}'} {\rm wt}\left(b_{\underline{i}^{\uparrow}}^{\underline{j}^{\uparrow}}\right),
\end{equation}
where the last equality is due to the fact that only $(\mathbf{t}',b,\mathbf{s}')$ can be consistently labelled by $\underline{i}^{\uparrow}$ and $\underline{j}^{\uparrow}$.  Similarly, $X^T v_{\underline{j}^{\uparrow}}= v_{\underline{j}}$ 
\begin{equation}  \label{eqn:case1_2}
-\left\langle f(v_{\underline{i}}),Xv_{\underline{j}^{\uparrow}}^*\right\rangle = \left\langle f(v_{\underline{i}}),(X^T v_{\underline{j}^{\uparrow}})^*\right\rangle = \left\langle f(v_{\underline{i}}),v_{\underline{j}}^*\right\rangle = \gamma_{\mathbf{s},b,\mathbf{t}} {\rm wt}\left(b_{\underline{i}}^{\underline{j}} \right).
\end{equation}
The weights in equations (\ref{eqn:case1_1}) and (\ref{eqn:case1_2}) are equal, since we did not change the parity of the label. Therefore we obtain $\gamma_{\mathbf{s},b,\mathbf{t}} = \gamma_{\mathbf{s}',b,\mathbf{t}'}$.

\textbf{Case II: Cup.}  In case $S$ is a cup, $\underline{i}$ contains no label $m+1$ and so $Xv_{\underline{i}}=0$, whereas $\underline{j}$ contains $m+1$ and $\overline{m+1}$ and therefore $X^T v_{\underline{j}{\downarrow}} = \pm ( v_{\underline{j}} + v_{\underline{j}^{\updownarrow}})$. Since  $\left\langle f(Xv_{\underline{i}}),v_{\underline{j}^{\downarrow}}^*\right\rangle = -\left\langle f(v_{\underline{i}}),Xv_{\underline{j}^{\downarrow}}^*\right\rangle$ we obtain

\begin{equation*}
0  = \left\langle f(v_{\underline{i}}),Xv_{\underline{j}^{\downarrow}}^*\right\rangle =\pm \left(\gamma_{\mathbf{s},b,\mathbf{t}} {\rm wt}\left(b_{\underline{i}}^{\underline{j}}\right) + \gamma_{\mathbf{s},b,\mathbf{t}'} {\rm wt}\left(b_{\underline{i}}^{{\underline{j}^{\updownarrow}}}\right) \right).
\end{equation*}
As we changed only the orientation at a cup with large label, the two weights will exactly differ by a sign and since $\mathbf{s} = \mathbf{s}'$ in this case, we have $\gamma_{\mathbf{s},b,\mathbf{t}} = \gamma_{\mathbf{s}',b,\mathbf{t}'}$.

\textbf{Case III: Cap.}
If $S$ is a cap then $\overline{m+1}$ appears twice in ${\underline{i}^{\uparrow}}$  whereas $m+1$ does not appear, so $f(v_{\underline{i}^{\uparrow}})=0$. To see this note that the label $m+1$ must correspond to vertical strands and therefore any consistently labeled diagram with labels ${\underline{i}^{\uparrow}}$ at the bottom must have more than $k$ vertical strands  and so the claim follows by the maximality of $k$.  Therefore,  $\left\langle X.f(v_{\underline{i}^{\uparrow}}),v_{\underline{j}}^*\right\rangle=0$. On the other hand $X.v_{\underline{i}\uparrow} = \pm\left( v_{\underline{i}} + v_{{\underline{i}^{\updownarrow}}} \right)$, thus
\begin{eqnarray*}
0  = \pm \left\langle f(v_{\underline{i}} + v_{\underline{i}^{\updownarrow}}),v_{\underline{j}}^*\right\rangle = \pm \left( \gamma_{\mathbf{s}',b,\mathbf{t}} {\rm wt}\left(b_{\underline{i}}^{\underline{j}}\right) + \gamma_{\mathbf{s},b,\mathbf{t}} {\rm wt}\left(b_{\underline{i}^{\updownarrow}}^{\underline{j}}\right) \right).
\end{eqnarray*}
We now switched the orientation of a cap with large label and so the two weights differ by a sign. Furthermore $\mathbf{t} = \mathbf{t}'$. Thus we obtain $\gamma_{\mathbf{s},b,\mathbf{t}} = \gamma_{\mathbf{s}',b,\mathbf{t}'}$.

\textit{Assume ${\rm sdim V} = 2m+1|0$ and $d \leq m=m+n$:}
Let $(\mathbf{t},b,\mathbf{s})$ be our  oriented Brauer diagram. Thanks to the assumption on $d$ we can pick $\underline{i} \in {\rm Vect}(\mathbf{s})$, $\underline{j} \in {\rm Vect}(\mathbf{t})$ to get a consistent labelling of $(\mathbf{t},b,\mathbf{s})$ with the following property
\begin{center}
{\it Different strands are labelled with different absolute values and there is a unique strand $S$ labelled with absolute value $1$.}
\end{center}
Let $X$ be the unique element in $\mathfrak{osp}(V)$ that maps $v_{1}$ to $v_{0}$, $v_{0}$ to $-v_{\overline{1}}$ and all other basis elements to zero. In the presentation from Definition \ref{osp} this means that only $-u_1^T$ contains a non-zero entry, namely a $1$ as leftmost entry. Let $(\mathbf{t}^\circ,b^\circ,\mathbf{s}^\circ)$ be the unique oriented Brauer diagram obtained from $(\mathbf{t},b,\mathbf{s})$ by  deleting the strand $S$ and replacing it by two singleton sets. Define $\mathbf{t}^\circ$ and $\mathbf{t}^\circ$ accordingly and let $\underline{i}^\circ\in {\rm Vect}(\mathbf{s}^\circ)$ and $\underline{j}^\circ\in {\rm Vect}(\mathbf{t}^\circ)$ be the consistent labelling of $(\mathbf{t}^\circ,b^\circ,\mathbf{s}^\circ)$ obtained from $\underline{i}$ and $\underline{j}$ by changing all $1$'s resp. $\overline{1}$'s into $0$'s. 

We distinguish again three cases:

\textbf{Case I: Vertical Strand.} We assume that $S$ is labelled $1$, the case of label $\overline{1}$ is done analogously. 
Clearly, $Xv_{\underline{i}}= v_{\underline{i}^\circ}$, hence
\begin{equation} \label{eqn:odd_case1_1}
\left\langle f(Xv_{\underline{i}}),v_{{\underline{j}^{\circ}}}^*\right\rangle = \left\langle f(v_{\underline{i}^\circ}),v_{\underline{j}^\circ}^*\right\rangle = \gamma_{\mathbf{s}^\circ,b^\circ,\mathbf{t}^\circ} {\rm wt}\left((b^\circ)_{\underline{i}^\circ}^{\underline{j}^\circ}\right),
\end{equation}
where the last equality is due to the fact that only $(\mathbf{t}^\circ,b^\circ,\mathbf{s}^\circ)$ can be consistently labelled by $\underline{i}^\circ$ and $\underline{j}^\circ$. On the other hand we can use equivariance of $f$ and the fact that $X^T v_{\underline{j}^\circ}= v_{\underline{j}}$ to obtain
\begin{equation}  \label{eqn:odd_case1_2}
-\left\langle f(v_{\underline{i}}),Xv_{\underline{j}^\circ}^*\right\rangle = \left\langle f(v_{\underline{i}}),v_{\underline{j}}^*\right\rangle = \gamma_{\mathbf{s},b,\mathbf{t}} {\rm wt}\left(b_{\underline{i}}^{\underline{j}} \right).
\end{equation}
Since $v_1$ has even parity, the weights in equations (\ref{eqn:odd_case1_1}) and (\ref{eqn:odd_case1_2}) are equal. Therefore we obtain $\gamma_{\mathbf{s},b,\mathbf{t}} = \gamma_{\mathbf{s}^\circ,b^\circ,\mathbf{t}^\circ}$. 

\textbf{Case II: Cup.}  Hence, $\underline{i}$ contains no $1$ and so $Xv_{\underline{i}}=0$. Note that $\mathbf{s} = \mathbf{s}^\circ$. Now let $\underline{j''}$ be the sequence obtained from $\underline{j}$ by switching the unique $1$ to $0$. It holds that $X^T v_{\underline{j}''} =  v_{\underline{j}} - v_{\underline{j}^\circ}$ and thus 
\begin{eqnarray*}
0 = \left\langle f(Xv_{\underline{i}}),v_{\underline{j}''}^*\right\rangle = -\left\langle f(v_{\underline{i}}),Xv_{\underline{j}''}^*\right\rangle = \gamma_{\mathbf{s},b,\mathbf{t}} {\rm wt}\left(b_{\underline{i}}^{\underline{j}}\right) - \gamma_{\mathbf{s},b^\circ,\mathbf{t}^\circ} {\rm wt}\left((b^\circ)_{\underline{i}}^{\underline{j}^\circ}\right).
\end{eqnarray*}

Since we changed a cup with small label, the two appearing weights agree and thus and we obtain $\gamma_{\mathbf{s},b,\mathbf{t}} = \gamma_{\mathbf{s}^\circ,b^\circ,\mathbf{t}^\circ}$, since $\mathbf{s} = \mathbf{s}^\circ$. 

\textbf{Case III: Cap.} Note that $\underline{i}^\circ$ is the labelling derived from $\underline{i}$ by replacing $1$ and $\overline{1}$ by $0$.  Let $\underline{i}''$ be obtained from $\underline{i}$ by switching the unique $\overline{1}$ to $0$,
With the same argument as in Case III above,  the maximality of  $k$  implies $f(v_{\underline{i}''})=0$. Hence also $\left\langle X.f(v_{\underline{i}''}),v_{\underline{j}}^*\right\rangle=0$. On the other hand $X.v_{\underline{i}''} = v_{\underline{i}^\circ} - v_{\underline{i}}$, and therefore we obtain
\begin{eqnarray*}
0 = \left\langle f(Xv_{\underline{i}''}),v_{\underline{j}}^*\right\rangle = \left\langle f(v_{\underline{i^\circ}} - v_{\underline{i}}),v_{\underline{j}}^*\right\rangle = \gamma_{\mathbf{s}^\circ,b^\circ,\mathbf{t}} {\rm wt}\left((b^\circ)_{\underline{i^\circ}}^{\underline{j}}\right) - \gamma_{\mathbf{s},b,\mathbf{t}} {\rm wt}\left(b_{\underline{i}}^{\underline{j}}\right).
\end{eqnarray*}
using the fact that  $\mathbf{t} = \mathbf{t}^\circ$. As above, we only changed a cap with small labels so the weights agree. Thus we obtain $\gamma_{\mathbf{s},b,\mathbf{t}} = \gamma_{\mathbf{s}^\circ,b^\circ,\mathbf{t}^\circ}$. 

In all three cases we proved $\gamma_{\mathbf{s},b,\mathbf{t}} = \gamma_{\mathbf{s}^\circ,b^\circ,\mathbf{t}^\circ}$. 
The same arguments apply after we switched the orientation on the strand $S$, hence $\gamma_{\mathbf{s}',b,\mathbf{t}'} = \gamma_{\mathbf{s}^\circ,b^\circ,\mathbf{t}^\circ}$. This implies the claim $\gamma_{\mathbf{s},b,\mathbf{t}} = \gamma_{\mathbf{s}',b,\mathbf{t}'}$.

\textit{Assume: ${\rm sdim V} = 2m|0$ and $d < m=m+n$:} This case is very similar to the previous one, but easier. We just replace the occurrences of the label $0$ by $2$ in the previous argument. For instance we choose $X$ to be the unique element in $\mathfrak{osp}(V)$ that maps $v_{1}$ to $v_{\overline{2}}$, $v_{2}$ to $-v_{\overline{1}}$ and all other basis elements to zero (in the presentation from Definition \ref{osp} this means that only the matrix $a_1$ contains two non-zero entries). Furthermore we assume that no strand is labelled with absolute value $2$, which is possible by the assumption $d < m+n$. Then the calculations are the same as in the previous case.
\end{proof}

We are left with showing that Claim (1) from the proof of Theorem \ref{thm:maintwo} holds in the odd case.

\begin{lemma} \label{lem:find_larger}
Assume $d \leq m+n$. Let $f \in {\rm End}_{{\rm Mat}(\mathcal{OB}_d(m-n))}\left( \bigoplus_{\mathbf{s}} \mathbf{s}\right)$ be an element that commutes with the action of $\mathfrak{osp}(V)$ and write
$$ f = \sum_{(\mathbf{t}, b, \mathbf{s}) \text{ oriented}} \gamma_{\mathbf{s},b,\mathbf{t}} 1_\mathbf{s} b 1_\mathbf{t}.$$
Assume that there exist $\gamma_{\mathbf{s},b,\mathbf{t}} \neq 0$ such that $b \in \widehat{\mathcal{B}}[d]_{(k,r)}$ for some $r > 0$. Then there exists $\gamma_{\mathbf{s}',b',\mathbf{t}'} \neq 0$ such that $b' \in \widehat{\mathcal{B}}[d]_{(k,r-1)}$ and $b \lhd b'$.
\end{lemma}
\begin{proof}
Note that the assumptions of the lemma are never satisfied in the even case, so we can assume that we are in the odd case.
Let $\underline{i} \in {\rm Vect}({\mathbf{s}})$ and $\underline{j} \in {\rm Vect}({\mathbf{t}})$ be such that $b_{\underline{i}}^{\underline{j}}$ is consistently labelled. In addition we assume that all strands in $b$ are labelled by pairwise different absolute values and none of them is labelled with the absolute value $1$. This is possible due to the assumption $d \leq m+n$ and the assumption that there are less than $d$ strands, since $r > 0$. By the definition of the action we know that the coefficient of $v_{\underline{j}}$ in $F(f) v_{\underline{i}}$, i.e., $\left\langle F(f) v_{\underline{i}} , v_{\underline{j}}^* \right\rangle$, is equal to $\gamma_{b,\mathbf{s},\mathbf{t}} {\rm wt}(b_{\underline{i}}^{\underline{j}})$, since this is the only diagram that can be consistently coloured by $\underline{i}$ and $\underline{j}$.

We first assume that $\mathbf{s}$ contains the symbol $\circ$ at least once. Then let $\underline{i}' \in {\rm Vect}(\mathbf{s}(\underline{i}'))$ be equal to $\underline{i}$ except that the first occurring $0$ is changed into a $1$, furthermore let $X \in \mathfrak{osp}(V)$ be the unique element mapping $v_0$ to $v_{\overline{1}}$, $v_1$ to $-v_0$ and send all other basis elements to zero. Due to the assumption that $f$ commutes with $X$ we have the equalities
$$ \left\langle F(f)X.v_{\underline{i}'}, v_{\underline{j}}^* \right\rangle = \left\langle F(f)v_{\underline{i}'}, X.v_{\underline{j}}^* \right\rangle = \left\langle F(f)v_{\underline{i}'}, (-X^T.v_{\underline{j}})^* \right\rangle.$$
Computing the very left side we obtain $X.v_{\underline{i}'} = -v_{\underline{i}} + \sum_{\underline{k}} v_{\underline{k}}$, where the sum runs over all those $\underline{k}$ that differ from $\underline{i}'$ by switching one $0$ into a $\overline{1}$. Then
$$ \left\langle F(f)X.v_{\underline{i}'}, v_{\underline{j}}^* \right\rangle = -\gamma_{\mathbf{s},b,\mathbf{t}} {\rm wt}\left(b_{\underline{i}}^{\underline{j}}\right) + \sum_{\underline{k}} \gamma_{\mathbf{s}(\underline{k}),c(\underline{k}),\mathbf{t}} {\rm wt}\left(c(\underline{k})_{\underline{k}}^{\underline{j}}\right),$$
where $c(\underline{k})$ is the unique diagram where the $1$ and $\overline{1}$ in $\underline{k}$ are connected by a horizontal arc and $\underline{k} \in {\rm Vect}(\mathbf{s}(\underline{k}))$. For all $c(\underline{k})$ it holds that $b \lhd c(\underline{k})$ and they contain one more strand than $b$.

To compute the right side we see that $-X^T.v_{\underline{j}} = - \sum_{\underline{l}} v_{\underline{l}}$, where the sum runs over those $\underline{l}$ that differ from $\underline{j}$ by switching exactly one $0$ into a $1$. Hence 
$$ \left\langle F(f)v_{\underline{i}'}, X.v_{\underline{j}}^* \right\rangle = -\sum_{\underline{l}} \gamma_{\mathbf{s}(\underline{i}'),d(\underline{l}),\mathbf{t}(\underline{l})} {\rm wt}\left(d(\underline{l})_{\underline{i}'}^{\underline{l}}\right),$$
where $d(\underline{l})$ is the unique diagram where the $1$'s in both $\underline{i}'$ and $\underline{l}$ are connected by a vertical strand and $\underline{l} \in {\rm Vect}(\mathbf{t}(\underline{l}))$. Again it holds for all $d(\underline{l})$ that $b \lhd d(\underline{l})$ and that they contain one more strand than $b$.

Putting the two sides together we obtain
$$\gamma_{\mathbf{s},b,\mathbf{t}} {\rm wt}\left(b_{\underline{i}}^{\underline{j}}\right) = \sum_{\underline{k}} \gamma_{\mathbf{s}(\underline{k}),c(\underline{k}),\mathbf{t}} {\rm wt}\left(c(\underline{k})_{\underline{k}}^{\underline{j}}\right) + \sum_{\underline{l}} \gamma_{\mathbf{s}(\underline{i}'),d(\underline{l}),\mathbf{t}(\underline{l})} {\rm wt}\left(d(\underline{l})_{\underline{i}'}^{\underline{l}}\right).$$
Since all the occurring weights are $\pm 1$ we see that if $\gamma_{\mathbf{s},b,\mathbf{t}} \neq 0$ then there must be at least one coefficient on the right that is also non-zero.

The case that $\mathbf{s}$ contains no $\circ$ is dealt with in an analogous fashion by swapping the roles of $\underline{i}$ and $\underline{j}$.
\end{proof}

From this it follows that we can always find a chain of diagrams with non-zero coefficients that end in an oriented (non-generalized) Brauer diagram.

\begin{corollary} \label{cor:eliminate_subdiagrams}
Assume $d \leq m+n$. Let $f \in {\rm End}_{{\rm Mat}(\mathcal{OB}_d(m-n))}\left( \bigoplus_{\mathbf{s}} \mathbf{s}\right)$ be an element that commutes with the action of $\mathfrak{osp}(V)$ and write
$$ f = \sum_{(\mathbf{t}, b, \mathbf{s}) \text{ oriented}} \gamma_{\mathbf{s},b,\mathbf{t}} 1_\mathbf{s} b 1_\mathbf{t}.$$
Assume that there exist $\gamma_{\mathbf{s}_r,b_r,\mathbf{t}_r} \neq 0$ such that $b_r \in \widehat{\mathcal{B}}[d]_{(k,r)}$ for some $r > 0$. 
Then there exists a sequence $\left( (\mathbf{s}_i,b_i,\mathbf{t}_i) \right)_{0 \leq i \leq r}$ of oriented generalized Brauer diagrams with $b_i \in \widehat{\mathcal{B}}[d]_{(k_i,i)}$ for some $k_i \geq k$ such that $\gamma_{\mathbf{s}_i,b_i,\mathbf{t}_i} \neq 0$ for all $i$ and $b_r \lhd b_{r-1} \lhd \ldots \lhd b_1 \lhd b_0$. Especially note that $b_0 \in \widehat{\mathcal{B}}[d]_{(k_0,0)}$.
\end{corollary}
\begin{proof}
This follows by successively applying Lemma \ref{lem:find_larger} to $(\mathbf{s}_i,b_i,\mathbf{t}_i)$ and setting $(\mathbf{s}_{i-1},b_{i-1},\mathbf{t}_{i-1}) := (\mathbf{s}',b',\mathbf{t}')$.
\end{proof}

\begin{remark}{\rm
In general, the map $\Psi$ is not surjective. In \cite{LZ1} it was shown that for $d \geq  m(2n+1)$ in the even case the map $\Psi$ is not surjective. This does not contradict Theorem \ref{thm:maintwo} since in the case of $m,n \neq 0$ we have $m(2n+1) = 2mn+m$ which is strictly greater than the assumed bound in the theorem and in the case of $n=0$ it holds $m(2n+1)=m$ which is also strictly greater. In the classical cases, the bounds are accurate for the symplectic and even orthogonal cases, but in the odd orthogonal case the Schur-Weyl duality holds already for $2m+1\geq d$,  see \cite[p. 870 Theorem b)]{Brauer} and \cite[Theorem 1.4]{Grood}, whereas our bound is $m\geq d$.}
\end{remark}

\begin{remark}{\rm
Similar to the definition of the oriented Brauer category $\mathcal{OB}(m-n)$, one defines the Brauer category $\mathcal{B}r (\delta)$ using Brauer diagrams having different numbers of vertices at top and bottom. Both, the embedding into the additive closure of the oriented Brauer category as well as the proof of surjectivity generalise to this situation, as they only rely on the diagrammatic description. The number $d$ has to be substituted for the number of arcs in each morphism space.}
\end{remark}

\bibliographystyle{spmpsci}

\end{document}